\documentclass[a4paper,reqno,10pt]{amsart}
\makeatletter
\newcommand*{\rom}[1]{\expandafter\@slowromancap\romannumeral #1@}
\makeatother
\usepackage{epsf,graphicx,epsfig}
\usepackage{amscd}
\usepackage{amsmath,latexsym,amssymb,amsthm}
\usepackage{fdsymbol}
\usepackage[nospace,noadjust]{cite}
\usepackage{textcomp}
\usepackage{setspace,cite}
\usepackage{mathrsfs,amsmath}
\usepackage{chemarrow}
\usepackage{lscape,fancyhdr,fancybox}
\usepackage{stmaryrd}
\usepackage[all,cmtip]{xy}
\usepackage{tikz-cd}
\usepackage{cancel}
\usepackage[usestackEOL]{stackengine}
\usetikzlibrary{shapes,arrows,decorations.markings}
\usepackage{graphicx}

\usepackage{color}
\usepackage{bbm, dsfont}
\usepackage{xcolor}
\usepackage{url}
\usepackage{enumerate}
\usepackage[mathscr]{euscript}

  \theoremstyle{plain}
\swapnumbers
    \newtheorem{thm}{Theorem}[section]
    \newtheorem{prop}[thm]{Proposition}

    \newtheorem{subsec}[thm]{}
\theoremstyle{definition}
    \newtheorem{defn}[thm]{Definition}
        \newtheorem{remark}[thm]{Remark}
    \newtheorem{exam}[thm]{Example}

\theoremstyle{remark}

\title{}
\author{}
\date{}
\usepackage{amssymb}

\usepackage{hyperref}
\hypersetup{
colorlinks,
citecolor=blue,
filecolor=black,
linkcolor=blue,
urlcolor=black
}

\usepackage[left=25mm,right=25mm,top=25mm,bottom=25mm,paper=a4paper]{geometry}

\begin{document}


\title[Factorizations, classifying complements problem and deformation maps]{Factorizations, classifying complements problem and deformation maps for Lie-Yamaguti algebras}



\author{Apurba Das}
\address{Department of Mathematics,
Indian Institute of Technology, Kharagpur 721302, West Bengal, India.}
\email{apurbadas348@gmail.com, apurbadas348@maths.iitkgp.ac.in}

\begin{abstract}
A Lie-Yamaguti algebra is a non-associative algebraic structure that generalizes both Lie algebras and Lie triple systems. We first consider the factorization problem for Lie-Yamaguti algebras that essentially related to the bicrossed product of Lie-Yamaguti algebras. Next, given an inclusion $\mathfrak{g} \subset E$ of Lie-Yamaguti algebras and a strong $\mathfrak{g}$-complement $\mathfrak{h}$, we describe and classify all $\mathfrak{g}$-complements in $E$. In particular, we show that any other $\mathfrak{g}$-complement in $E$ is isomorphic to $\mathfrak{h}$ by some deformation map $r: \mathfrak{h} \rightarrow \mathfrak{g}$. Despite this importance, it turns out that a deformation map generalizes homomorphisms, derivations, crossed homomorphisms and relative Rota-Baxter operators on Lie-Yamaguti algebras. We define the cohomology of a deformation map unifying the cohomologies of all the operators mentioned above. Finally, we provide a Maurer-Cartan characterization and construct the governing $L_\infty$-algebra of a deformation map $r$ that controls the linear deformations of $r$.
\end{abstract}

\maketitle

\medskip

\medskip

\begin{center}
\end{center}

\thispagestyle{empty}

\tableofcontents

\vspace{0.2cm}

\section{Introduction}\label{section-1}
\subsection{Lie triple systems and Lie-Yamaguti algebras} The notion of {\em Lie triple systems} can be traced back to Cartan's work on symmetric spaces. Such triple systems were first formulated in their algebraic form and named as Lie triple systems by Jacobson \cite{jacobson}. Subsequently, Lister \cite{lister} established a structure theory of Lie triple systems. Later, various studies on Lie triple systems were developed in connection with other algebraic structures. In a few years, Nomizu \cite{nomizu} considered the tangent space of a reductive homogeneous space at a point together with the torsion and curvature tensors associated with a certain canonical affine connection. The underlying algebraic structure was considered by Yamaguti \cite{yamaguti0, yamaguti} by the name of {\em general Lie triple systems}. This new structure is given by a binary operation and a ternary operation that generalizes both Lie algebras and Lie triple systems. In his papers, Yamaguti also considered representations and cohomology of general Lie triple systems. Various other studies on this new structure were subsequently made by Sagle \cite{sagle,sagle2} and Kikkawa \cite{kikkawa}. In \cite{kinyon}, Kinyon and Weinstein observed that general Lie triple systems, which they called {\em Lie-Yamaguti algebras} in their paper, can be obtained from Leibniz algebras. They also examine the enveloping Lie algebras of Lie-Yamaguti algebras. Irreducible Lie-Yamaguti algebras and their relations with orthogonal Lie algebras are explored by Benito et al. \cite{benito0,benito1}. Recently, Takahashi \cite{takahashi} considered modules over quandles using representations of Lie-Yamaguti algebras. Various operators (e.g., homomorphisms, automorphisms,  derivations, crossed homomorphisms, and Rota-Baxter operators) on Lie-Yamaguti algebras and related Yamaguti-type algebraic structures have also been studied recently in connection with the cohomology theory \cite{goswami,meyberg,mondal-saha,sheng-zhao,sun-chen,zhao-qiao,zhao-qiao2,zhao-xu-qiao,das-ms}. Many important results on Lie algebras and Lie triple systems are still developing in the context of Lie-Yamaguti algebras.

\subsection{Factorization problem, classifying complements problem and deformation maps}
The {\em factorization problem} is a very traditional question that can be asked for various algebraic structures. In the context of Lie algebras, this problem asks for describing and classifying all Lie algebra structures on the direct sum $E= \mathfrak{g} \oplus \mathfrak{h}$ such that $\mathfrak{g}$ and $\mathfrak{h}$ are both Lie subalgebras of $E$. The answer to this problem can be traced back to the bicrossed products of $\mathfrak{g}$ and $\mathfrak{h}$ \cite{majid}. The {\em classifying complements problem} is the converse of the factorization problem, and this was first considered by Agore and Militaru \cite{agore-lie} in the context of Lie algebras. Later, it was generalized to the category of other algebraic structures \cite{agore-ass,agore-leibniz}. In the context of Lie algebras, this problem concerns identifying and classifying certain Lie subalgebras that complement a given Lie subalgebra inside a larger Lie algebra. Let $\mathfrak{g} \subset E$ be an inclusion of Lie algebras. A $\mathfrak{g}$-complement in $E$ is a Lie subalgebra $\mathfrak{h} \subset E$ such that $ E = \mathfrak{g} + \mathfrak{h}$ and $\mathfrak{g} \cap \mathfrak{h} = \{ 0 \}$. Then the classifying complements problem asks to find and classify all $\mathfrak{g}$-complements in $E$. It is known that if $\mathfrak{h}$ is a $\mathfrak{g}$-complement in $E$ then $E$ is isomorphic to the bicrossed product $\mathfrak{g} \Join \mathfrak{h}$ associated with a matched pair of $\mathfrak{g}$ and $\mathfrak{h}$ (called the canonical matched pair). 
It has been observed in \cite{agore-lie} that given a $\mathfrak{g}$-complement $\mathfrak{h}$, any other $\mathfrak{g}$-complement in $E$ can be obtained from $\mathfrak{h}$ by some {\em deformation map} $r: \mathfrak{h} \rightarrow \mathfrak{g}$ in the canonical matched pair. The notion of deformation map also attracts independent interest as it generalizes Lie algebra homomorphisms, derivations, crossed homomorphisms, Rota-Baxter operators of weight $0$ and $1$. In \cite{jiang-sheng-tang}, Jiang, Sheng and Tang discovered the cohomology of a deformation map $r$ (in the context of Lie algebras) unifying the well-known cohomologies of all the operators mentioned above. Additionally, they have constructed the governing algebra of a deformation map $r$ that controls the linear deformations of the operator $r$.

\subsection{Main results and the outline of the paper} This paper is primarily devoted to studying the factorization problem, the classifying complements problem, and deformation maps in the context of Lie-Yamaguti algebras. We begin by recalling the notion of a matched pair of Lie-Yamaguti algebras considered in \cite{zhao-qiao2}. Subsequently, we describe the corresponding bicrossed product and the associated factorization problem. Since a Lie-Yamaguti algebra has a ternary operation (that disturbs the question asked in the traditional factorization problem), we need some `strong' assumption in the factorization problem with respect to the trilinear operation to get its correspondence with the bicrossed product of Lie-Yamaguti algebras (cf. Theorem \ref{thm-factorization}). Next, we introduce the notion of a deformation map in a given matched pair $(\mathfrak{g}, \mathfrak{h})$ of Lie-Yamaguti algebras. A deformation map unifies homomorphisms, derivations, crossed homomorphisms, and relative Rota-Baxter operators on Lie-Yamaguti algebras \cite{sun-chen,zhao-qiao,zhao-qiao2,zhao-xu-qiao}. We observe that a linear map $r: \mathfrak{h} \rightarrow \mathfrak{g}$ is a deformation map if and only if its graph ${Graph} (r)$ is a Lie-Yamaguti subalgebra of the associated bicrossed product. We also show that a deformation map $r$ induces a new Lie-Yamaguti algebra structure on the vector space $\mathfrak{h}$ (we denote this induced structure by $\mathfrak{h}_r$). Moreover, there is a representation of this induced Lie-Yamaguti algebra $\mathfrak{h}_r$ on the vector space $\mathfrak{g}$. Then we move to the classifying complements problem. Given an inclusion $\mathfrak{g} \subset E$ of Lie-Yamaguti algebras and a `strong' $\mathfrak{g}$-complement $\mathfrak{h}$, there is a matched pair $(\mathfrak{g}, \mathfrak{h})$ of Lie-Yamaguti algebras (called the canonical matched pair) for which $E \cong \mathfrak{g} \Join \mathfrak{h}$ as Lie-Yamaguti algebras. Subsequently, we show that, if $r: \mathfrak{h} \rightarrow \mathfrak{g}$ is a deformation map in the canonical matched pair then ${Graph} (r)$ is also a $\mathfrak{g}$-complement in $E$. Moreover, any $\mathfrak{g}$-complement in $E$ arises in this way (cf. Theorem \ref{thm-ccp}). We consider an equivalence relation $\sim$ on the set $\mathcal{DM} (\mathfrak{g}, \mathfrak{h})$ of all deformation maps and show that the set $\mathcal{DM} (\mathfrak{g}, \mathfrak{h}) / \sim$ has a bijection with the set of all isomorphism classes of $\mathfrak{g}$-complements in $E$ (cf. Theorem \ref{thm-ccp2}). We also discuss the factorization problem and the classifying complements problem to the particular context of Lie triple systems.

\medskip

Despite the importance of deformation maps in the classifying complements problem, it is worth meaningful to study deformation maps independently in the context of Lie-Yamaguti algebras. We first define the cohomology of a deformation map that unifies the existing cohomologies of various operators on Lie-Yamaguti algebras \cite{sun-chen,zhao-qiao,zhao-qiao2,zhao-xu-qiao}. Next, given a matched pair of Lie-Yamaguti algebras, we construct an $L_\infty$-algebra whose {Maurer-Cartan elements} are precisely deformation maps (cf. Theorem \ref{thm-mc-char}). In particular, we recover the Maurer-Cartan characterization of Rota-Baxter operators of weight $0$ \cite{zhao-qiao} and obtain characterizations of other operators. Finally, for a given deformation map $r$, we construct a new $L_\infty$-algebra (called the governing algebra that is obtained by twisting the corresponding Maurer-Cartan element) that governs the linear deformations of the operator $r$ (cf. Theorem \ref{thm-governing}).


\medskip

The paper is organized as follows. In Section \ref{section-2}, we recall Lie-Yamaguti algebras and their cohomology with coefficients in a representation. Then matched pair of Lie-Yamaguti algebras, associated bicrossed product, and the factorization problem have been considered in Section \ref{section-3}. We introduce and study deformation maps in Section \ref{section-4}. In particular, here we consider the classifying complements problem and show that the set of all isomorphism classes of complements has a bijection with the set of all equivalence classes of deformation maps. Finally, in Section \ref{section-5}, we define the cohomology and the governing $L_\infty$-algebra of a deformation map by unifying the cohomologies of various well-known operators on Lie-Yamaguti algebras.

\medskip

All vector spaces, (multi)linear maps, unadorned tensor products, and wedge products are over a field ${\bf k}$ of characteristics $0.$

\section{Lie-Yamaguti algebras}\label{section-2} 

This section provides a brief review of the necessary background on Lie-Yamaguti algebras, including their representations and cohomology. For more details, we suggest the references \cite{kinyon,takahashi,yamaguti0,yamaguti,zhang-li,zhao-qiao}.

\begin{defn}
    A {\bf Lie-Yamaguti algebra} is a triple $(\mathfrak{g}, [~,~], \llbracket ~, ~, ~ \rrbracket)$ consisting of a vector space $\mathfrak{g}$ equipped with linear maps $[~,~] : \wedge^2 \mathfrak{g} \rightarrow \mathfrak{g}$ and $\llbracket ~, ~, ~ \rrbracket : \wedge^2 \mathfrak{g} \otimes \mathfrak{g} \rightarrow \mathfrak{g}$ that satisfy the following set of identities:
    \begin{equation}
         [ [x, y], z] + [[y, z], x] + [[ z, x], y] + \llbracket x, y, z \rrbracket + \llbracket y, z, x \rrbracket + \llbracket z, x, y \rrbracket = 0,
         \end{equation}
         \begin{equation}
        \llbracket [x, y], z, w \rrbracket + \llbracket [y, z], x , w \rrbracket + \llbracket [z, x], y, w \rrbracket = 0,
         \end{equation}
         \begin{equation}\label{eqn-3}
         \llbracket x, y , [z, w ] \rrbracket = [ \llbracket x, y, z \rrbracket,  w] + [z, \llbracket x, y, w \rrbracket ],
          \end{equation}
         \begin{equation}
        \llbracket x, y, \llbracket z, w, t \rrbracket \rrbracket = \llbracket \llbracket x, y, z \rrbracket, w, t \rrbracket + \llbracket z, \llbracket x, y, w \rrbracket, t \rrbracket + \llbracket z, w, \llbracket x, y, t \rrbracket \rrbracket, \label{ly4}
    \end{equation}
    for all $x, y, z, w, t \in \mathfrak{g}$. We often denote a Lie-Yamaguti algebra as above simply by the space $\mathfrak{g}$ when the defining operations are clear from the context.
\end{defn}

There are three important classes of Lie-Yamaguti algebras, which are described below.

\begin{exam}\label{exam-lie-ly}
    Let $(\mathfrak{g}, [~,~])$ be a Lie algebra. If we take $\llbracket ~, ~, ~ \rrbracket$ to be the zero map, then $(\mathfrak{g}, [~,~], \llbracket ~, ~, ~ \rrbracket)$ is trivially a Lie-Yamaguti algebra. Instead, if we set 
    \begin{align}\label{lie-ly}
        \llbracket x, y, z \rrbracket := [ [ x, y], z], \text{ for } x, y, z \in \mathfrak{g}
    \end{align}
    then $(\mathfrak{g}, [~,~], \llbracket ~, ~, ~ \rrbracket)$ is also a Lie-Yamaguti algebra.
\end{exam}

\begin{exam}
    A {\bf Lie triple system} is a pair $(\mathfrak{g}, \llbracket ~, ~, ~ \rrbracket)$ of a vector space $\mathfrak{g}$ with a linear map $\llbracket ~, ~, ~ \rrbracket : \wedge^2 \mathfrak{g} \otimes \mathfrak{g} \rightarrow \mathfrak{g}$ that satisfy
    \begin{align*}
         \llbracket x, y, z \rrbracket + \llbracket y, z, x \rrbracket + \llbracket z, x, y \rrbracket = 0, \text{ for all } x, y, z \in \mathfrak{g}
    \end{align*}
    and the identity (\ref{ly4}). It follows that a Lie triple system $(\mathfrak{g}, \llbracket ~, ~, ~ \rrbracket)$ can be regarded as a Lie-Yamaguti algebra with the trivial bracket $[~,~]$.
\end{exam}

\begin{exam}
    Let $(\mathfrak{l}, \circ )$ be a left Leibniz algebra. That is, $\mathfrak{l}$ is a vector space endowed with a linear map $\circ : \mathfrak{l} \otimes \mathfrak{ l} \rightarrow \mathfrak{l}$ satisfying
    \begin{align*}
         x \circ (y \circ z) = (x \circ y) \circ z + y \circ (x \circ z), \text{ for } x, y, z \in \mathfrak{l}.
    \end{align*}
    Then the triple $(\mathfrak{l}, [~,~], \llbracket ~, ~, ~ \rrbracket)$ is a Lie-Yamaguti algebra, where
    \begin{align}\label{leib-ly}
        [x, y] := x \circ y - y \circ x \quad \text{ and } \quad \llbracket x, y, z \rrbracket := - (x \circ y) \circ z, \text{ for } x, y, z \in \mathfrak{l}.
    \end{align}
\end{exam}

\medskip

Recall that a {\bf representation} of a Lie-Yamaguti algebra $(\mathfrak{g}, [~,~ ], \llbracket ~, ~, ~ \rrbracket)$ is a vector space $V$ endowed with two linear maps $\rho : \mathfrak{g} \rightarrow \mathrm{End}(V),  \!~ x \mapsto \rho_x$ and $\mu : \mathfrak{g} \otimes \mathfrak{g} \rightarrow \mathrm{End}(V), ~ x \otimes y \mapsto \mu (x, y)$ that satisfy the following conditions:
\begin{equation*}
    \mu ([x, y], z) = \mu (x, z) ~ \! \rho_y - \mu (y, z) ~ \! \rho_x,
    \end{equation*}
    \begin{equation*}
     \mu (x, [y, z]) = \rho_y ~ \! \mu (x, z) - \rho_z ~ \! \mu (x, y), 
      \end{equation*}
    \begin{equation*}
      \rho_{ \llbracket  x, y, z \rrbracket } =  D_{\rho, \mu} (x, y) ~ \! \rho_z - \rho_z ~ \! D_{\rho, \mu} (x, y),
       \end{equation*}
    \begin{equation*}
       \mu(z, w) \mu (x, y) - \mu (y, w) \mu(x, z) - \mu (x, \llbracket  y, z, w \rrbracket) + D_{\rho, \mu} (y, z) \mu (x, w) = 0,
        \end{equation*}
    \begin{equation*}
        \mu ( \llbracket x, y, z \rrbracket, w) + \mu (z, \llbracket x, y, w \rrbracket) = D_{\rho, \mu} (x, y)  \mu(z, w) - \mu(z, w) D_{\rho, \mu} (x, y),
\end{equation*}
for all $x, y, z, w \in \mathfrak{g}$.
Here the linear map $D_{\rho, \mu} : \wedge^2 \mathfrak{g} \rightarrow \mathrm{End}(V)$ is given by
\begin{align*}
    D_{\rho, \mu} (x, y) :=  \rho_x  \rho_y - \rho_y \rho_x - \rho_{[x, y]} - \mu (x, y) + \mu(y, x), \text{ for } x, y \in \mathfrak{g}.
\end{align*}
We denote a representation as above by the triple $(V; \rho, \mu)$ or simply by $V$ when the structure maps $\rho$ and $\mu$ are understood.

\begin{exam}
    Let $(\mathfrak{g}, [~,~], \llbracket ~, ~, ~ \rrbracket)$ be any arbitrary Lie-Yamaguti algebra. We define linear maps $\rho_\mathrm{ad} : \mathfrak{g} \rightarrow \mathrm{End} (\mathfrak{g})$ and $\mu_\mathrm{ad} : \mathfrak{g} \otimes \mathfrak{g} \rightarrow \mathrm{End}(\mathfrak{g})$ by
    \begin{align*}
        (\rho_\mathrm{ad})_x y := [x, y] \quad \text{ and } \quad (\mu_\mathrm{ad})(x, y) z := \llbracket z, x, y \rrbracket, \text{ for } x, y, z \in \mathfrak{g}.
    \end{align*}
    Then $(\mathfrak{g};  \rho_\mathrm{ad} , \mu_\mathrm{ad})$ is a representation of the given Lie-Yamaguti algebra, called the {\em adjoint representation}. For this representation, the map $D_{   \rho_\mathrm{ad} , \mu_\mathrm{ad} } : \wedge^2 \mathfrak{g} \rightarrow \mathrm{End} (\mathfrak{g})$ is simply given by $ D_{   \rho_\mathrm{ad} , \mu_\mathrm{ad} }  (x, y) z = \llbracket x, y, z \rrbracket$, for $x, y, z \in \mathfrak{g}$. 
\end{exam}

See \cite{sheng-zhao,takahashi} for some more examples of Lie-Yamaguti algebra representations. Given a Lie-Yamaguti algebra and a representation of it, one may construct the corresponding {\em semidirect product}, which turns out to be a Lie-Yamaguti algebra \cite{sheng-zhao}. We do not discuss it here, as in the next section, we recall the more general {\em bicrossed product} associated with a matched pair of Lie-Yamaguti algebras.


\medskip

Let $(\mathfrak{g}, [~,~], \llbracket ~, ~, ~ \rrbracket)$ be a Lie-Yamaguti algebra and $(V; \rho, \mu)$ be any representation of it. For each $n \geq 1$, we set
\begin{align*}
   C^n (\mathfrak{g}, V) :=  \begin{cases}
    \mathrm{Hom} (\mathfrak{g}, V) & \text{ if } n = 1,\\
        \mathrm{Hom} \big( \underbrace{(\wedge^2 \mathfrak{g}) \otimes \cdots \otimes (\wedge^2 \mathfrak{g})}_{(n-1) \text{~copies}} , V \big) \oplus \mathrm{Hom} \big(  \underbrace{(\wedge^2 \mathfrak{g}) \otimes \cdots \otimes (\wedge^2 \mathfrak{g})}_{(n-1) \text{~copies}} \otimes~\! \mathfrak{g} , V \big) & \text{ if } n \geq 2.
    \end{cases}
\end{align*}
Thus, it follows that any element $F \in C^{n \geq 2} (\mathfrak{g}, V)$ has two components $F = (F_\mathrm{I}, F_\mathrm{II})$.
Then there is a map $\delta : C^n (\mathfrak{g}, V) \rightarrow C^{n+1} (\mathfrak{g}, V)$ explicitly given by
\begin{align*}
    (\delta (f))_\mathrm{I} ( x, y) =~& \rho_x f(y) - \rho_y f(x) - f ([x, y]), \\
    (\delta (f))_\mathrm{II} (x, y, z) =~& D_{\rho, \mu} (x, y) f(z) + \mu (y, z) f (x) - \mu (x, z) f (y) - f (\llbracket x, y, z \rrbracket),
\end{align*}
for $f \in C^1 (\mathfrak{g}, V) = \mathrm{Hom} (\mathfrak{g}, V)$ and $x, y, z \in \mathfrak{g}$, and
\begin{align*}
    &(\delta (F))_\mathrm{I} (\mathfrak{X}_1, \ldots, \mathfrak{X}_n) \\
    &= (-1)^{n-1} \big\{  \rho_{x_n} F_{\mathrm{II}} (\mathfrak{X}_1, \ldots, \mathfrak{X}_{n-1}, y_n  ) - \rho_{y_n} F_\mathrm{II} (\mathfrak{X}_1, \ldots, \mathfrak{X}_{n-1}, x_n) - F_\mathrm{II} (\mathfrak{X}_1, \ldots, \mathfrak{X}_{n-1}, [x_n, y_n])  \big\} \\
    &\quad + \sum_{k=1}^{n-1} (-1)^{k+1} ~ \! D_{\rho, \mu} (x_k, y_k) F_\mathrm{I} (\mathfrak{X}_1, \ldots, \widehat{\mathfrak{X}_k}, \ldots, \mathfrak{X}_n) \\
    &\quad + \sum_{1 \leq i < j \leq n} (-1)^i ~ \! F_\mathrm{I} (\mathfrak{X}_1, \ldots, \widehat{\mathfrak{X}_i} , \ldots, [x_i, y_i, x_j] \wedge y_j + x_j \wedge [x_i, y_i, y_j], \ldots, \mathfrak{X}_n),
    \end{align*}
    \begin{align*}
    &(\delta (F))_\mathrm{II} (\mathfrak{X}_1, \ldots, \mathfrak{X}_n, x) \\
    &= (-1)^{n-1} \big\{  \mu (y_n, x) F_\mathrm{II} (\mathfrak{X}_1, \ldots, \mathfrak{X}_{n-1}, x_n) -  \mu (x_n, x) F_\mathrm{II} (\mathfrak{X}_1, \ldots, \mathfrak{X}_{n-1}, y_n)  \big\} \\
    & \quad + \sum_{k=1}^{n} (-1)^{k+1} ~ \! D_{\rho, \mu} (x_k, y_k) F_\mathrm{II} (\mathfrak{X}_1, \ldots, \widehat{\mathfrak{X}_k}, \ldots, \mathfrak{X}_n, x) \\
    & \quad + \sum_{1 \leq i < j \leq n} (-1)^i ~ \! F_\mathrm{II} (\mathfrak{X}_1, \ldots, \widehat{\mathfrak{X}_i} , \ldots, [x_i, y_i, x_j] \wedge y_j + x_j \wedge [x_i, y_i, y_j], \ldots, \mathfrak{X}_n, x) \\
    & \quad + \sum_{k=1}^n (-1)^k F_\mathrm{II} ( \mathfrak{X}_1, \ldots, \widehat{\mathfrak{X}_k}, \ldots, \mathfrak{X}_n, \llbracket x_k, y_k, x \rrbracket ),
\end{align*}
for $F = (F_\mathrm{I}, F_\mathrm{II}) \in C^{n \geq 2} (\mathfrak{g}, V)$, $\mathfrak{X}_i = x_i \wedge y_i \in \wedge^2 \mathfrak{g}$ for all $i =1, \ldots, n$ and $x \in \mathfrak{g}$.
Then it turns out that $\delta^2 = 0$ and thus $\{ C^\bullet (\mathfrak{g}, V), \delta \}$ is a cochain complex \cite{yamaguti}. The corresponding cohomology groups $H^\bullet (\mathfrak{g}, V)$ are called the cohomologies of the Lie-Yamaguti algebra $(\mathfrak{g}, [~,~], \llbracket ~, ~, ~ \rrbracket)$ with coefficients in the representation $(V; \rho, \mu)$.

\section{Matched pairs of Lie-Yamaguti algebras and the factorization problem}\label{section-3} This section begins by recalling the notion of a matched pair of Lie-Yamaguti algebras and the bicrossed product considered in \cite{zhao-qiao2}. In the end, we describe and study the factorization problem for Lie-Yamaguti algebras.

\begin{defn}\label{mply}
    A {\bf matched pair of Lie-Yamaguti algebras} is a tuple 
    \begin{align*}
        \big(  (\mathfrak{g}, [~,~], \llbracket ~, ~, ~ \rrbracket), (\mathfrak{h}, [~,~]', \llbracket ~, ~, ~ \rrbracket'), (\rho, \mu), (\psi, \nu)   \big)
    \end{align*}
    consisting of two Lie-Yamaguti algebras $(\mathfrak{g}, [~,~], \llbracket ~, ~, ~ \rrbracket)$ and $ (\mathfrak{h}, [~,~]', \llbracket ~, ~, ~ \rrbracket')$ equipped with four linear maps
    \begin{align*}
        \rho : \mathfrak{g} \rightarrow \mathrm{End}(\mathfrak{h}), \quad \mu: \mathfrak{g} \otimes \mathfrak{g} \rightarrow \mathrm{End}(\mathfrak{h}), \quad \psi : \mathfrak{h} \rightarrow \mathrm{End} (\mathfrak{g}) ~~~~ \text{ and } ~~~~~ \nu : \mathfrak{h} \otimes \mathfrak{h} \rightarrow \mathrm{End}(\mathfrak{g})
    \end{align*}
    such that (i) $(\mathfrak{h}; \rho, \mu)$ is a representation of the Lie-Yamaguti algebra  $(\mathfrak{g}, [~,~], \llbracket ~, ~, ~ \rrbracket),$ (ii) $(\mathfrak{g}; \psi, \nu)$ is a representation of the Lie-Yamaguti algebra $(\mathfrak{h}, [~,~]', \llbracket ~, ~, ~ \rrbracket')$,
and the following compatibility conditions hold:
    \begin{equation}\label{mp1}
        \rho_x ([\alpha, \beta]') = [\rho_x (\alpha) , \beta]' + [\alpha, \rho_x (\beta)]' + \rho_{\psi_\beta x} (\alpha) - \rho_{\psi_\alpha x} (\beta),
        \end{equation}
         \begin{equation}
           \rho_x (\llbracket \alpha, \beta, \gamma \rrbracket' ) = \llbracket \alpha, \beta, \rho_x (\gamma) \rrbracket' - \rho_{  D_{\psi, \nu} (\alpha, \beta) x} (\gamma), 
        \end{equation}
        \begin{equation}
           \llbracket \rho_x (\alpha), \beta, \gamma \rrbracket' + \llbracket \alpha, \rho_x (\beta), \gamma \rrbracket' = 0, 
        \end{equation}
        \begin{equation}
             \rho_{ \nu (\alpha, \beta) x} (\gamma ) = \rho_{ \nu (\alpha, \gamma) x} (\beta),
        \end{equation}
        \begin{equation}
          \mu (x, y) ([\alpha, \beta]' )= \mu (\psi_\beta (x), y) \alpha - \mu (\psi_\alpha (x), y) \beta,
        \end{equation}
           \begin{equation}
        \mu (x, y) \llbracket \alpha, \beta, \gamma \rrbracket' = \llbracket \alpha, \beta, \mu (x, y) \gamma \rrbracket' - \mu ( D_{\psi, \nu} (\alpha, \beta) x, y) \gamma - \mu (x, D_{\psi, \nu} (\alpha, \beta)y) \gamma,
    \end{equation}
        \begin{equation}
            \mu (x, \psi_\alpha (y)) \beta = [\alpha, \mu (x, y) \beta]',
        \end{equation}
      \begin{equation}
        \mu (\nu (\alpha, \beta) x, y) \gamma - \mu (\nu (\alpha, \gamma) x, y) \beta = \mu (x, D_{\psi, \nu} (\beta, \gamma) y) \alpha - \llbracket \beta, \gamma, \mu (x, y) \alpha \rrbracket',
    \end{equation}
    \begin{equation}
        \mu (x, \nu (\alpha, \beta) y) \gamma = \llbracket \mu (x, y) \gamma, \alpha, \beta \rrbracket' - D_{\rho, \mu } (y, \nu (\gamma, \alpha) x ) \beta + \mu ( y, \nu (\gamma, \beta) x) \alpha,
    \end{equation}
    \begin{equation}\label{mp10}
        \psi_\alpha ([x, y]) = [\psi_\alpha (x), y] + [x, \psi_\alpha (y)] + \psi_{\rho_y (\alpha)} x - \psi_{\rho_x (\alpha)} y,
    \end{equation}
    \begin{equation}\label{eqn-16}
        \psi_\alpha (\llbracket x, y, z \rrbracket) = \llbracket x, y, \psi_\alpha (z) \rrbracket - \psi_{D_{\rho, \mu} (x, y) \alpha } z,
    \end{equation}
    \begin{equation}
        \llbracket \psi_\alpha (x), y , z \rrbracket + \llbracket x, \psi_\alpha (y), z \rrbracket = 0,
    \end{equation}
     \begin{equation}
        \psi_{ \mu (x, y) \alpha} (z) = \psi_{ \mu (x, z) \alpha} (y),
    \end{equation}
    \begin{equation}
        \nu (\alpha, \beta) ([x, y] ) = \nu ( \rho_y (\alpha), \beta) x - \nu (\rho_x (\alpha), \beta) y,
    \end{equation}
    \begin{equation}
        \nu (\alpha, \beta) \llbracket x, y, z \rrbracket = \llbracket x, y, \nu (\alpha, \beta) z \rrbracket - \nu ( D_{\rho, \mu} (x, y) \alpha , \beta) z - \nu (\alpha , D_{\rho, \mu} (x, y) \beta) z,
    \end{equation}
    \begin{equation}
        \nu (\alpha, \rho_x (\beta)) y = [x, \nu (\alpha, \beta) y],
    \end{equation}
    \begin{equation}
        \nu (\mu (x, y) \alpha, \beta) z - \nu ( \mu (x, z) \alpha , \beta) y = \nu (\alpha, D_{\rho, \mu} (y, z) \beta) x - \llbracket y, z, \nu (\alpha, \beta) x \rrbracket,
    \end{equation}
    \begin{equation}\label{mp18}
        \nu (\alpha, \mu (x, y) \beta) z = \llbracket \nu (\alpha , \beta) z, x, y \rrbracket - D_{\psi, \nu} (\beta, \mu (z, x ) \alpha ) y + \nu (\beta, \mu (z, y) \alpha) x,
    \end{equation}
    for all $x, y, z \in \mathfrak{g}$ and $\alpha, \beta, \gamma \in \mathfrak{h}$. We often denote a matched pair of Lie-Yamaguti algebras as above simply by the tuple $( \mathfrak{g}, \mathfrak{h}, (\rho, \mu), (\psi, \nu))$ when the Lie-Yamaguti structures on $\mathfrak{g}$ and $\mathfrak{h}$ are understood.
\end{defn}

\begin{remark}
(i) The notion defined in Definition \ref{mply} is symmetric with respect to the Lie-Yamaguti algebras under consideration. More precisely, if  $ \big(  \mathfrak{g}, \mathfrak{h}, (\rho, \mu), (\psi, \nu)   \big)$  is a matched pair of Lie-Yamaguti algebras then  $ \big(\mathfrak{h},  \mathfrak{g},  (\psi, \nu) , (\rho, \mu) \big)$ is so. 

    (ii) (\cite{zhao-qiao2}) In a matched pair $ \big(  \mathfrak{g}, \mathfrak{h}, (\rho, \mu), (\psi, \nu)   \big)$ of Lie-Yamaguti algebras, the following identities also hold as consequences of (\ref{mp1})-(\ref{mp18}):
    \begin{equation}
        D_{\rho, \mu}(x, y) [\alpha, \beta]' = [ D_{\rho, \mu} (x, y) \alpha , \beta]' + [ \alpha, D_{\rho, \mu} (x, y) \beta]',
    \end{equation}
    \begin{equation}
        D_{\rho, \mu}(x, y) \llbracket \alpha, \beta, \gamma \rrbracket' = \llbracket D_{\rho, \mu}(x, y) \alpha, \beta, \gamma \rrbracket' + \llbracket \alpha, D_{\rho, \mu}(x, y) \beta, \gamma \rrbracket' + \llbracket \alpha, \beta, D_{\rho, \mu}(x, y) \gamma \rrbracket',
    \end{equation}
    \begin{equation}
        D_{\rho, \mu} (\psi_\alpha (x), y) + D_{\rho, \mu}(x, \psi_\alpha (y) ) = 0, 
    \end{equation}
    \begin{equation}
        D_{\psi, \nu} (\alpha, \beta) [x, y] = [ D_{\psi, \nu} (\alpha, \beta) x, y] + [x, D_{\psi, \nu} (\alpha, \beta) y],
    \end{equation}
    \begin{equation}
        D_{\psi, \nu} (\alpha, \beta) \llbracket x, y, z \rrbracket = \llbracket D_{\psi, \nu} (\alpha, \beta) x, y, z \rrbracket + \llbracket x, D_{\psi, \nu} (\alpha, \beta) y, z \rrbracket + \llbracket x, y, D_{\psi, \nu} (\alpha, \beta) z \rrbracket,
    \end{equation}
    \begin{equation}\label{mp24}
        D_{\psi, \nu} (\rho_x (\alpha), \beta ) + D_{\psi, \nu} (\alpha, \rho_x (\beta) ) = 0.
    \end{equation}
\end{remark}

The following result is well-known \cite{zhao-qiao2} and can be proved by direct calculation using (\ref{mp1})-(\ref{mp24}).

\begin{thm}\label{thm-bicrossed}
    Let $\big(  (\mathfrak{g}, [~,~], \llbracket ~, ~, ~ \rrbracket), (\mathfrak{h}, [~,~]', \llbracket ~, ~, ~ \rrbracket'), (\rho, \mu), (\psi, \nu)  \big)$ be a matched pair of Lie-Yamaguti algebras. Define linear maps $[~,~]_\Join : \wedge^2 (\mathfrak{g} \oplus \mathfrak{h}) \rightarrow \mathfrak{g} \oplus \mathfrak{h}$ and $\llbracket ~, ~, ~ \rrbracket_\Join : \wedge^2 (\mathfrak{g} \oplus \mathfrak{h}) \otimes (\mathfrak{g} \oplus \mathfrak{h}) \rightarrow \mathfrak{g} \oplus \mathfrak{h}$ by
    \begin{equation}
         [(x, \alpha), (y, \beta) ]_\Join := \big(  [x, y] + \psi_\alpha y - \psi_\beta x ~ \! , ~ \! [\alpha, \beta]' + \rho_x \beta - \rho_y \alpha \big), \label{bicrossed-1}
    \end{equation}
    \begin{align}
          \llbracket (x, \alpha), (y, \beta), (z, \gamma) \rrbracket_\Join :=~& \big( \llbracket x, y, z \rrbracket+ D_{\psi, \nu} (\alpha, \beta) z + \nu (\beta, \gamma) x - \nu (\alpha, \gamma) y  ~ \!, \nonumber \\
         & \qquad \llbracket \alpha, \beta, \gamma \rrbracket' + D_{\rho, \mu} (x, y) \gamma + \mu (y, z) \alpha - \mu (x, z) \beta \big), \label{bicrossed-2}
    \end{align}
    for $(x, \alpha), (y, \beta), (z, \gamma) \in \mathfrak{g} \oplus \mathfrak{h}$. Then the triple $(\mathfrak{g} \oplus \mathfrak{h}, [~,~]_\Join, \llbracket ~, ~, ~ \rrbracket_\Join)$ is a Lie-Yamaguti algebra.
\end{thm}

The Lie-Yamaguti algebra $(\mathfrak{g} \oplus \mathfrak{h}, [~,~]_\Join, \llbracket ~, ~, ~ \rrbracket_\Join)$ constructed in the above theorem is called the {\bf bicrossed product} of the given matched pair of Lie-Yamaguti algebras. We denote this bicrossed product algebra simply by $\mathfrak{g} \Join \mathfrak{h}$. It turns out from (\ref{bicrossed-1}) and (\ref{bicrossed-2}) that both $\mathfrak{g}$ and $\mathfrak{h}$ are Lie-Yamaguti subalgebras of the bicrossed product $\mathfrak{g} \Join \mathfrak{h}$. 


\begin{exam}\label{direct-ly}
    Let $\mathfrak{g}$ and $\mathfrak{h}$ be two Lie-Yamaguti algebras. Then $ \big( \mathfrak{g}, \mathfrak{h}, (\rho = 0, \mu = 0 ), (\psi = 0, \nu = 0) \big)$ is a matched pair of Lie-Yamaguti algebras. The corresponding bicrossed product is precisely the direct product $(\mathfrak{g} \oplus \mathfrak{h}, [~,~]_\mathrm{dir} , \llbracket ~, ~, ~ \rrbracket_\mathrm{dir})$, where for $(x, \alpha), (y, \beta), (z, \gamma) \in \mathfrak{g} \oplus \mathfrak{h},$
    \begin{align*}
        [(x, \alpha), (y, \beta)]_\mathrm{dir} := ([x, y] ~ \! \!, ~ \! \! [\alpha, \beta]') ~~~~ \text{ and } ~~~~ \llbracket (x, \alpha), (y, \beta), (z, \gamma) \rrbracket_\mathrm{dir} := ( \llbracket x, y, z \rrbracket ~ \! \! , ~ \! \! \llbracket \alpha, \beta, \gamma \rrbracket' ).
    \end{align*}
    Symmetrically, $(\mathfrak{h}, \mathfrak{g}, (\psi = 0, \nu = 0), (\rho = 0, \mu = 0 ))$ is a matched pair of Lie-Yamaguti algebras.
\end{exam}

\begin{exam}\label{semidirect-ly}
    Let $(\mathfrak{g}, [~,~], \llbracket ~, ~, ~ \rrbracket)$ be a Lie-Yamaguti algebra and $(V; \rho, \mu)$ be a representation of it. We regard $V$ as a trivial Lie-Yamaguti algebra. Then $( \mathfrak{g}, V, (\rho, \mu), (\psi= 0, \nu = 0))$ is a matched pair of Lie-Yamaguti algebras. The corresponding bicrossed product is precisely the semidirect product $(\mathfrak{g} \oplus V, [~,~]_\ltimes, \llbracket ~, ~, ~ \rrbracket_\ltimes)$, where for $(x, \alpha), (y, \beta), (z, \gamma) \in \mathfrak{g} \oplus V$,
    \begin{align*}
        [ (x, \alpha) , (y, \beta)]_\ltimes :=~& ([x, y] ~ \! , ~ \! \rho_x \beta - \rho_y \alpha ),\\
        \llbracket   (x, \alpha), (y, \beta), (z, \gamma) \rrbracket_\ltimes :=~& \big(   \llbracket x, y, z \rrbracket ~ \! , ~ \! D_{\rho, \mu} (x, y) \gamma + \mu (y, z) \alpha - \mu (x, z) \beta \big).
    \end{align*}
    Symmetrically, the tuple $(  V,  \mathfrak{g},  (\psi= 0, \nu = 0), (\rho, \mu)  )$ is also a matched pair of Lie-Yamaguti algebras.
\end{exam}

\begin{exam}\label{action-ly}
    Let $(\mathfrak{g}, [~,~], \llbracket ~, ~ , ~ \rrbracket)$ and  $(\mathfrak{h}, [~,~]', \llbracket ~, ~ , ~ \rrbracket' )$ be two Lie-Yamaguti algebras. A {\bf Lie-Yamaguti action} of $\mathfrak{g}$ on $\mathfrak{h}$ is given by linear maps $\rho: \mathfrak{g} \rightarrow \mathrm{End} (\mathfrak{h})$ and $\mu : \mathfrak{g} \otimes \mathfrak{g} \rightarrow \mathrm{End} (\mathfrak{h})$ that make the triple $(\mathfrak{h}; \rho, \mu)$ into a representation of the Lie-Yamaguti algebra $(\mathfrak{g}, [~,~], \llbracket ~, ~, ~ \rrbracket)$ satisfying additionally
    \begin{equation*}
        \rho_x ([\alpha, \beta]') = [\rho_x (\alpha), \beta]' + [\alpha, \rho_x (\beta)]',
        \end{equation*}
    \begin{equation*}
        \rho_x ( \llbracket \alpha, \beta, \gamma \rrbracket') = \llbracket \alpha, \beta, \rho_x (\gamma) \rrbracket', \quad \llbracket \rho_x (\alpha), \beta, \gamma \rrbracket' + \llbracket \alpha, \rho_x (\beta), \gamma \rrbracket' = 0,
    \end{equation*}
    \begin{equation*}
        \mu (x, y) ([\alpha, \beta]') = 0, \quad \mu (x, y) ( \llbracket \alpha, \beta, \gamma \rrbracket') = 0,
    \end{equation*}
    \begin{equation*}
        [ \mu (x, y) \alpha, \beta]' = 0 ~~~~ \text{ and } ~~~~ \llbracket \mu (x, y) \alpha, \beta, \gamma \rrbracket' = \llbracket \alpha, \beta, \mu (x, y) \gamma \rrbracket' = 0,
    \end{equation*}
    for all $x, y \in \mathfrak{g}$ and $\alpha, \beta, \gamma \in \mathfrak{h}.$
    Then it turns out that $\big( \mathfrak{g}, \mathfrak{h}, (\rho, \mu), (\psi = 0 , \nu = 0) \big)$ is a matched pair of Lie-Yamaguti algebras. Thus, symmetrically, the tuple  $\big( \mathfrak{h}, \mathfrak{g}, (\psi = 0 , \nu = 0) ,   (\rho, \mu)   \big)$ is so.
\end{exam}

The bicrossed product of Lie-Yamaguti algebras given in Theorem \ref{thm-bicrossed} is related to the {\em factorization problem}. Let $(\mathfrak{g}, [~,~], \llbracket ~, ~, ~ \rrbracket)$ and $(\mathfrak{h}, [~,~]', \llbracket ~, ~, ~ \rrbracket')$ be two Lie-Yamaguti algebras. Then a new Lie-Yamaguti algebra $(E, [~, ~]_E, \llbracket ~, ~, ~ \rrbracket_E )$ is said to {\bf factorize} through $\mathfrak{g}$ and $\mathfrak{h}$ if 

(i) $E$ contains $\mathfrak{g}$ and $\mathfrak{h}$ as Lie-Yamaguti subalgebras,

(ii) $E = \mathfrak{g} + \mathfrak{h}$ and $\mathfrak{g} \cap \mathfrak{h} = \{ 0 \}$.

\noindent It is said to be {\bf strongly factorize} if additionally $\llbracket \mathfrak{g}, \mathfrak{h}, \mathfrak{h} \rrbracket_E \subset \mathfrak{g}$ and $ \llbracket \mathfrak{h}, \mathfrak{g}, \mathfrak{g} \rrbracket_E \subset \mathfrak{h}$. Then the factorization problem (in the context of Lie-Yamaguti algebras) asks to describe and classify all Lie-Yamaguti algebras that strongly factorize through $\mathfrak{g}$ and $\mathfrak{h}$.

\begin{thm}\label{thm-factorization}
    Let $(\mathfrak{g}, [~,~], \llbracket ~, ~, ~ \rrbracket)$ and $(\mathfrak{h}, [~,~]', \llbracket ~, ~, ~ \rrbracket')$ be two Lie-Yamaguti algebras. Then a Lie-Yamaguti algebra $(E, [~, ~]_E, \llbracket ~, ~, ~ \rrbracket_E )$ strongly factorizes through $\mathfrak{g}$ and $\mathfrak{h}$ if and only if there exists a matched pair of Lie-Yamaguti algebras $( (\mathfrak{g}, [~,~], \llbracket ~, ~, ~ \rrbracket),  (\mathfrak{h}, [~,~]', \llbracket ~, ~, ~ \rrbracket'), (\rho, \mu), (\psi, \nu))$ such that $E \cong \mathfrak{g} \Join \mathfrak{h}$ as Lie-Yamaguti algebras.
\end{thm}

\begin{proof}
    Suppose there is a matched pair $( (\mathfrak{g}, [~,~], \llbracket ~, ~, ~ \rrbracket),  (\mathfrak{h}, [~,~]', \llbracket ~, ~, ~ \rrbracket'), (\rho, \mu), (\psi, \nu))$ of Lie-Yamaguti algebras. Then Theorem \ref{thm-bicrossed} implies that the bicrossed product $(\mathfrak{g} \oplus \mathfrak{h}, [~, ~]_\Join, \llbracket ~, ~, ~  \rrbracket_\Join)$ strongly factorizes through $\mathfrak{g} \cong \mathfrak{g} \oplus \{ 0 \}$ and $\mathfrak{h} \cong \{ 0 \} \oplus \mathfrak{h}$. Conversely, we assume that a Lie-Yamaguti algebra $(E, [~,~]_E, \llbracket ~, ~, ~ \rrbracket_E)$ strongly factorizes through $\mathfrak{g}$ and $\mathfrak{h}$. Since $E = \mathfrak{g} + \mathfrak{h}$ and $\mathfrak{g} \cap \mathfrak{h} = \{  0\}$, any element of $E$ can be uniquely written as a pair $(x, \alpha)$, for $x \in \mathfrak{g}$ and $\alpha \in \mathfrak{h}$. We now define linear maps $\rho : \mathfrak{g} \rightarrow \mathrm{End} (\mathfrak{h})$, $\psi : \mathfrak{h} \rightarrow \mathrm{End} (\mathfrak{g})$, $\mu : \mathfrak{g} \otimes \mathfrak{g} \rightarrow \mathrm{End} (\mathfrak{h})$ and $\nu : \mathfrak{h} \otimes \mathfrak{h} \rightarrow \mathrm{End} (\mathfrak{g})$ by
    \begin{align}
        [ (x, 0), (0, \alpha )]_E =~& (- \psi_\alpha x ~ \!, ~ \! \rho_x \alpha), \label{equa1}\\
        \llbracket (0, \alpha), (x, 0), (y, 0) \rrbracket_E =~& (0, \mu (x, y) \alpha),\\
        \llbracket (x, 0), (0, \alpha), (0, \beta) \rrbracket_E =~& (\nu (\alpha, \beta) x, 0), \label{equa3}
    \end{align}
    for $x, y \in \mathfrak{g}$ and $\alpha, \beta \in \mathfrak{h}$. It is easy to see that the Lie-Yamaguti algebra identities of $(E, [~,~]_E, \llbracket ~, ~, ~ \rrbracket_E)$ are equivalent to mean that the tuple 
   $( (\mathfrak{g}, [~,~], \llbracket ~, ~, ~ \rrbracket),  (\mathfrak{h}, [~,~]', \llbracket ~, ~, ~ \rrbracket'), (\rho, \mu), (\psi, \nu))$ is a matched pair of Lie-Yamaguti algebras. Additionally, in view of (\ref{equa1})-(\ref{equa3}), the Lie-Yamaguti algebra structure $(E, [~,~]_E, \llbracket ~, ~, ~ \rrbracket_E)$ can be identified with the bicrossed product $\mathfrak{g} \Join \mathfrak{h}$ of the above matched pair of Lie-Yamaguti algebras.
\end{proof}

When one considers Lie-Yamaguti algebras $(E, [~, ~]_E, \llbracket ~, ~, ~ \rrbracket_E)$ that only factorize through the Lie-Yamaguti algebras $\mathfrak{g}$ and $\mathfrak{h}$ (not necessarily strongly factorize), the above result doesn't hold. To observe this, we first recall that a {\bf matched pair of Lie algebras} \cite{agore-lie,majid} is a quadruple $( (\mathfrak{g}, [~,~]), (\mathfrak{h}, [~,~]'), \rho, \psi)$ that consists of two Lie algebras $(\mathfrak{g}, [~,~])$ and $ (\mathfrak{h}, [~,~]')$ together with Lie algebra representations $\rho : \mathfrak{g} \rightarrow \mathrm{End} (\mathfrak{h})$ and $\psi : \mathfrak{h} \rightarrow \mathrm{End} (\mathfrak{g})$ satisfying additionally the identities (\ref{mp1}) and (\ref{mp10}). In this case, $(\mathfrak{g} \oplus \mathfrak{h} , [~, ~]_\Join)$ is a Lie algebra, where the Lie bracket $[~, ~]_\Join : \wedge^2 (\mathfrak{g} \oplus \mathfrak{h}) \rightarrow \mathfrak{g} \oplus \mathfrak{h}$ is given by (\ref{bicrossed-1}). Hence by (\ref{lie-ly}), the triple $E = (\mathfrak{g} \oplus \mathfrak{h}, [~, ~]_\Join, \llbracket ~, ~, ~ \rrbracket_\Join)$ is a Lie-Yamaguti algebra, where
\begin{align*}
    \llbracket (x, \alpha), (y, \beta), (z, \gamma) \rrbracket_\Join := [ [(x, \alpha), (y, \beta)]_\Join, (z, \gamma) ]_\Join, \text{ for } (x, \alpha), (y, \beta), (z, \gamma) \in \mathfrak{g} \oplus \mathfrak{h}.
\end{align*}
The Lie-Yamaguti algebra $E$ obviously factorizes through the Lie-Yamaguti algebras $\mathfrak{g}$ and $\mathfrak{h}$ (here the Lie-Yamaguti structures on $\mathfrak{g}$ and $\mathfrak{h}$ are obtained from the respective Lie algebra structures by applying (\ref{lie-ly})). However, the pair $(\mathfrak{g}, \mathfrak{h})$ of Lie-Yamaguti algebras doesn't form a matched pair.

\medskip

Similarly, a {\bf matched pair of Leibniz algebras} \cite{agore-leibniz} is a tuple $( (\mathfrak{l}, \circ), (\mathfrak{m}, \circ'), (\rho^L, \rho^R), (\psi^L, \psi^R))$ consisting of Leibniz algebras $(\mathfrak{l}, \circ)$ and  $ (\mathfrak{m}, \circ')$ together with linear maps
    \begin{align*}
        \rho^L : \mathfrak{l} \otimes \mathfrak{m} \rightarrow \mathfrak{m}, \quad \rho^R : \mathfrak{m} \otimes \mathfrak{l} \rightarrow \mathfrak{m}, \quad \psi^L : \mathfrak{m} \otimes \mathfrak{l} \rightarrow \mathfrak{l} ~~~~~ \text{ and } ~~~~~ \psi^R : \mathfrak{l} \otimes \mathfrak{m} \rightarrow \mathfrak{l}
    \end{align*}
    such that $( \mathfrak{m} ; \rho^L, \rho^R)$ is a representation of the Leibniz algebra $(\mathfrak{l}, \circ)$, and $( \mathfrak{l} ; \psi^L, \psi^R)$ is a representation of the Leibniz algebra $(\mathfrak{m}, \circ')$ satisfying a list of compatibility conditions that ensures that the direct sum $\mathfrak{l} \oplus \mathfrak{m}$ with the operation
    \begin{align*}
        (x, \alpha) \circ_\Join (y, \beta) = (x \circ y + \psi^L (\alpha, y) + \psi^R (x, \beta) ~ \! , ~ \! \alpha \circ' \beta + \rho^L (x, \beta) + \rho^R (\alpha, y) ),
    \end{align*}
    for $(x, \alpha), (y, \beta) \in \mathfrak{l} \oplus \mathfrak{m}$, is a Leibniz algebra. Hence the triple $E= (\mathfrak{l} \oplus \mathfrak{m}, [~,~]_\Join, \llbracket ~, ~, ~ \rrbracket_\Join)$ is a Lie-Yamaguti algebra, where
    \begin{align*}
        [ (x, \alpha), (y, \beta)]_\Join :=~& (x, \alpha) \circ_\Join (y, \beta) - (y, \beta) \circ_\Join (x, \alpha), \\
        \llbracket (x, \alpha), (y, \beta) , (z, \gamma) \rrbracket_\Join :=~& -\big(   (x, \alpha) \circ_\Join (y, \beta) \big) \circ_\Join (z, \gamma),
    \end{align*}
    for $(x, \alpha), (y, \beta), (z, \gamma) \in \mathfrak{l} \oplus \mathfrak{m}$. This Lie-Yamaguti algebra $E$ factorizes through the Lie-Yamaguti algebras $\mathfrak{l}$ and $\mathfrak{m}$ (that are obtained from the corresponding Leibniz algebras by applying (\ref{leib-ly})). However, the pair $(\mathfrak{l}, \mathfrak{m})$ of Lie-Yamaguti algebras doesn't form a matched pair.

    \medskip

    Let $(\mathfrak{g}, [~,~], \llbracket ~, ~, ~ \rrbracket)$ and $(\mathfrak{h}, [~,~]', \llbracket ~, ~, ~ \rrbracket')$ be two Lie-Yamaguti algebras. Suppose there are two matched pairs of Lie-Yamaguti algebras $(\mathfrak{g}, \mathfrak{h}, (\rho, \mu), (\psi, \nu))$ and $(\mathfrak{g}, \mathfrak{h}, (\rho', \mu'), (\psi', \nu'))$ with the associated bicroseed products $\mathfrak{g} \Join \mathfrak{h}$ and $\mathfrak{g} \Join' \mathfrak{h}$. Then there exists a bijective correspondence between the set of all Lie-Yamaguti algebra isomorphisms $\Phi: \mathfrak{g} \Join \mathfrak{h} \rightarrow \mathfrak{g} \Join' \mathfrak{h}$ for which $\mathfrak{g}$ and $\mathfrak{h}$ are both invariants (i.e., $\Phi (\mathfrak{g}) \subset \mathfrak{g}$ and $\Phi (\mathfrak{h}) \subset \mathfrak{h}$) and pairs $(u : \mathfrak{g} \rightarrow \mathfrak{g} ~ \! ,~ \! v: \mathfrak{h} \rightarrow \mathfrak{h})$ of Lie-Yamaguti algebra isomorphisms satisfying
    \begin{align}\label{lastt}
        &v (\rho_x \alpha) = \rho'_{u (x)} v (\alpha), \quad v (\mu (x, y) \alpha) = \mu' ( u(x), u(y)) v (\alpha), \quad u (\psi_\alpha x) = \psi'_{v (\alpha)} u(x) \nonumber \\
        & \qquad \text{ and } u (\nu (\alpha, \beta) x) = \nu' (v (\alpha), v (\beta)) u (x), \text{ for any } x, y \in \mathfrak{g} \text{ and } \alpha , \beta \in \mathfrak{h}. 
    \end{align}
    Given two Lie-Yamaguti algebras $\mathfrak{g}$ and $\mathfrak{h}$ as above, let $\mathcal{MP} (\mathfrak{g}, \mathfrak{h})$ be the set of all quadruples $(\rho, \mu, \psi, \nu)$ of linear maps that make $(\mathfrak{g}, \mathfrak{h}, (\rho, \mu) , (\psi, \nu))$ into a matched pair of Lie-Yamaguti algebras. A quadruple $(\rho, \mu, \psi, \nu)$ is equivalent to another such quadruple $(\rho', \mu', \psi', \nu')$ if there exist Lie-Yamaguti algebra isomorphisms $u : \mathfrak{g} \rightarrow \mathfrak{g}$ and $v : \mathfrak{h} \rightarrow \mathfrak{h}$ satisfying the four identities in (\ref{lastt}). This defines an equivalence relation $\sim$ on the set $\mathcal{MP} (\mathfrak{g}, \mathfrak{h})$. It follows that the set of all isomorphism classes of Lie-Yamaguti algebras that strongly factorize through $\mathfrak{g}$ and $\mathfrak{h}$, and make them invariant has a bijection with the set $\mathcal{MP} (\mathfrak{g}, \mathfrak{h}) / \sim$.

\section{Deformation maps and the classifying complements problem}\label{section-4} 

In this section, we first introduce deformation maps in a matched pair of Lie-Yamaguti algebras. Such a map unifies various well-known operators (e.g., homomorphisms, derivations, crossed homomorphisms and relative Rota-Baxter operators) on Lie-Yamaguti algebras. We show that a deformation map induces a new Lie-Yamaguti algebra structure, and there is a suitable representation of it. Next, given an inclusion $\mathfrak{g} \subset E$ of Lie-Yamaguti algebras and a strong $\mathfrak{g}$-complement $\mathfrak{h}$, we show that any other $\mathfrak{g}$-complement can be obtained by the graph of a deformation map $r : \mathfrak{h} \rightarrow \mathfrak{g}$. We also consider an equivalence relation $\sim$ on the set $\mathcal{DM} (\mathfrak{g}, \mathfrak{h})$ of all deformation maps and prove that the set of all isomorphism classes of $\mathfrak{g}$-complements has a bijection with $\mathcal{DM} (\mathfrak{g}, \mathfrak{h}) / \sim$.

\begin{defn}
 Let $\big(  \mathfrak{g}, \mathfrak{h}, (\rho, \mu), (\psi, \nu)   \big)$ be a matched pair of Lie-Yamaguti algebras. Then a linear map $r : \mathfrak{h} \rightarrow \mathfrak{g}$ is said to be a {\bf deformation map} if for all $\alpha, \beta, \gamma \in \mathfrak{h}$,
 \begin{align}\label{defor-1}
     [r(\alpha), r (\beta)] + \psi_\alpha r (\beta) - \psi_\beta r (\alpha) =~&  r \big( [\alpha, \beta]' + \rho_{r (\alpha)} \beta - \rho_{r (\beta)} \alpha \big),
     \end{align}
     \begin{align}\label{defor-2}
     \llbracket r (\alpha), r (\beta), r (\gamma) \rrbracket + D_{\psi, \nu} (\alpha, \beta) r(\gamma) & + \nu (\beta, \gamma) r (\alpha) - \nu (\alpha , \gamma) r (\beta) \nonumber \\
     = r \big(  \llbracket \alpha , \beta, \gamma \rrbracket' +& D_{\rho, \mu} (r (\alpha), r (\beta)) \gamma + \mu ( r (\beta), r (\gamma)) \alpha - \mu (r (\alpha), r (\gamma)) \beta \big).
 \end{align}
 We denote the set of all deformation maps by $\mathcal{DM} (\mathfrak{g}, \mathfrak{h}).$
\end{defn}

In the following, we first observe that a deformation map in a matched pair of Lie-Yamaguti algebras unifies several well-known operators 
on Lie-Yamaguti algebras.

\begin{exam}
    Let  $(\mathfrak{g}, [~,~], \llbracket ~, ~, ~ \rrbracket)$ and $(\mathfrak{h}, [~,~]', \llbracket ~, ~, ~ \rrbracket' )$ be two Lie-Yamaguti algebras. A linear map $\varphi : \mathfrak{g} \rightarrow \mathfrak{h}$ is said to be a {\bf homomorphism} of Lie-Yamaguti algebras from $\mathfrak{g}$ to $\mathfrak{h}$ if it satisfies 
    \begin{align*}
         \varphi ([x, y]) = [ \varphi (x), \varphi (y) ]' ~~~~ \text{ and } ~~~~ \varphi (\llbracket x, y, z \rrbracket ) = \llbracket \varphi (x), \varphi (y), \varphi (z) \rrbracket', \text{ for all } x, y, z \in \mathfrak{g}.
    \end{align*}
    It is easy to see that a linear map $\varphi : \mathfrak{g} \rightarrow \mathfrak{h}$ is a homomorphism of Lie-Yamaguti algebras if and only if $\varphi$ is a deformation map in the matched pair $(\mathfrak{h}, \mathfrak{g}, (\psi= 0 , \nu = 0), (\rho = 0, \mu = 0))$ considered in Example \ref{direct-ly}. 
    
\end{exam}

\begin{exam}
    Let $(\mathfrak{g}, [~,~], \llbracket ~, ~, ~ \rrbracket)$ be a Lie-Yamaguti algebra and $(V; \rho, \mu)$ be a representation of it.

    (i) A linear map $d : \mathfrak{g} \rightarrow V$ is said to be a {\bf derivation} (\cite{sun-chen}) on $\mathfrak{g}$ with values in $V$ if 
    \begin{align*}
        d ([x, y]) = \rho_x d (y) - \rho_y d (x) ~~~~ \text{ and } ~~~~ 
        d (\llbracket x, y, z \rrbracket ) =D_{\rho, \mu} (x, y) d (z) + \mu (y, z) d (x) - \mu (x, z) d (y),
    \end{align*}
    for all $x, y, z \in \mathfrak{g}$. It can be checked that a linear map $d : \mathfrak{g} \rightarrow V$ is a derivation if and only if $d $ is a deformation map in the matched pair $(V, \mathfrak{g}, (\psi= 0, \nu= 0), (\rho, \mu))$ considered in Example \ref{semidirect-ly}.

    (ii) A linear map $R : V \rightarrow \mathfrak{g}$ is said to be a {\bf relative Rota-Baxter operator of weight $0$} (also called an {\bf $\mathcal{O}$-operator} \cite{sheng-zhao,zhao-qiao}) if it satisfies
    \begin{align*}
        [R(\alpha), R(\beta)] =~& R (\rho_{R(\alpha)} \beta - \rho_{R(\beta)} \alpha), \\
        \llbracket R(\alpha), R(\beta), R (\gamma) \rrbracket =~& R \big(   D _{\rho, \mu} (R(\alpha), R(\beta)) \gamma + \mu (R(\beta), R(\gamma)) \alpha - \mu (R (\alpha), R(\gamma)) \beta   \big),
    \end{align*}
    for all $\alpha, \beta, \gamma \in V$. Then it follows that a linear map $R: V \rightarrow \mathfrak{g}$ is a relative Rota-Baxter operator of weight $0$ if and only if $R$ is a deformation map in the matched pair $(\mathfrak{g}, V, (\rho, \mu), (\psi = 0 , \nu = 0))$ considered in Example \ref{semidirect-ly}.
\end{exam}

\begin{exam}
    Let $(\mathfrak{g}, [~,~], \llbracket ~, ~, ~ \rrbracket)$ and $(\mathfrak{h}, [~,~]', \llbracket ~, ~, ~ \rrbracket')$ be two Lie-Yamaguti algebras, and $\rho, \mu$ defines a Lie-Yamaguti action of $\mathfrak{g}$ on $\mathfrak{h}$ (see Example \ref{action-ly}). 

    (i) Then a linear map $d : \mathfrak{g} \rightarrow \mathfrak{h}$ is said to be a {\bf crossed homomorphism} (also called a {\bf differential operator of weight $1$} \cite{sun-chen}) on $\mathfrak{g}$ with values in $\mathfrak{h}$ if it satisfies
    \begin{align*}
        d ([x, y]) =~&  [ d(x), d(y)]' + \rho_x d (y) - \rho_y d (x),\\
        d (\llbracket x, y, z \rrbracket ) =~& \llbracket d(x), d(y), d (z) \rrbracket' + D_{\rho, \mu} (x, y) d(z) + \mu (y, z ) d (x) - \mu (x, z) d (y) ,
    \end{align*}
    for all $x, y, z \in \mathfrak{g}$. It is easy to see that a linear map $d : \mathfrak{g} \rightarrow \mathfrak{h}$ is a crossed homomorphism if and only if $d$ is a deformation map in the matched pair $(\mathfrak{h}, \mathfrak{g}, (\psi= 0, \nu = 0), (\rho, \mu))$ considered in Example \ref{action-ly}.

    (ii) On the other hand, a linear map $R : \mathfrak{h} \rightarrow \mathfrak{g}$ is said to be a {\bf relative Rota-Baxter operator of weight $1$} if it satisfies
    \begin{align*}
        [R(\alpha), R(\beta)] =~& R ([\alpha, \beta]' + \rho_{R(\alpha)} \beta - \rho_{R(\beta)} \alpha), \\
        \llbracket R(\alpha), R(\beta), R (\gamma) \rrbracket =~& R \big( \llbracket \alpha, \beta, \gamma \rrbracket' + D _{\rho, \mu} (R(\alpha), R(\beta)) \gamma + \mu (R(\beta), R(\gamma)) \alpha - \mu (R (\alpha), R(\gamma)) \beta   \big),
    \end{align*}
    for all $\alpha, \beta, \gamma \in \mathfrak{h}$. It follows that a linear map $R: \mathfrak{h} \rightarrow \mathfrak{g}$ is a relative Rota-Baxter operator of weight $1$ if and only if $R$ is a deformation map in the matched pair $(\mathfrak{g}, \mathfrak{h}, (\rho, \mu), (\psi= 0, \nu = 0))$ considered in Example \ref{action-ly}.
\end{exam}

\medskip

Let $((\mathfrak{g}, [~,~]),( \mathfrak{h}, [~,~]'), \rho, \psi)$ be a matched pair of Lie algebras. Then a linear map $r : \mathfrak{h} \rightarrow \mathfrak{g}$ is called a {\em deformation map} \cite{agore-lie} if it satisfies
\begin{align}\label{lie-deformation}
    [r(\alpha), r (\beta)] + \psi_\alpha r (\beta) - \psi_\beta r (\alpha) =~&  r \big( [\alpha, \beta]' + \rho_{r (\alpha)} \beta - \rho_{r (\beta)} \alpha \big),
    \end{align}
    for all $\alpha, \beta \in \mathfrak{h}$. Then $d$ is also a deformation map in the matched pair
    \begin{align*}
        \big(  (\mathfrak{g}, [~, ~], \llbracket ~, ~, ~ \rrbracket = 0), (\mathfrak{h}, [~,~]', \llbracket ~, ~, ~ \rrbracket' = 0), (\rho, \mu = 0), (\psi, \nu = 0)    \big)
    \end{align*}
    of Lie-Yamaguti algebras. However, when the ternary operations of the Lie-Yamaguti structures on $\mathfrak{g}$ and $\mathfrak{h}$ are obtained from the respective Lie algebra structures by applying (\ref{lie-ly}), the pair $(\mathfrak{g}, \mathfrak{h})$ of Lie-Yamaguti algebras doesn't form a matched pair. Hence, deformation maps has no sense here. 


We return to deformation maps in an arbitrary matched pair of Lie-Yamaguti algebras. Here, we first provide a characterization of deformation maps in terms of their graph.

\begin{prop}\label{prop-graph}
     Let $\big( \mathfrak{g}, \mathfrak{h}, (\rho, \mu), (\psi, \nu)   \big)$ be a matched pair of Lie-Yamaguti algebras. Then a linear map $r : \mathfrak{h} \rightarrow \mathfrak{g}$ is a deformation map if and only if its graph
     \begin{align*}
         Graph (r) = \{  (r (\alpha), \alpha) ~ \! | ~ \! \alpha \in \mathfrak{h} \} 
     \end{align*}
     is a Lie-Yamaguti subalgebra of the bicrossed product $ \mathfrak{g} \Join \mathfrak{h} = (\mathfrak{g} \oplus \mathfrak{h}, [~, ~ ]_\Join, \llbracket ~, ~, ~ \rrbracket_\Join )$ constructed in Theorem \ref{thm-bicrossed}. 
\end{prop}

The proof of the above proposition is simple, and thus we omit it here. As a consequence of this result, we obtain the following.

\begin{thm}\label{thm-induced}
    Let $\big(  \mathfrak{g}, \mathfrak{h}, (\rho, \mu), (\psi, \nu)   \big)$ be a matched pair of Lie-Yamaguti algebras and $r : \mathfrak{h} \rightarrow \mathfrak{g}$ be a deformation map on it. 
    
    (i) Then the vector space $\mathfrak{h}$ inherits a new Lie-Yamaguti algebra structure $(\mathfrak{h}, [~,~]'_r , \llbracket ~, ~, ~ \rrbracket'_r)$ whose operations are given by
    \begin{align*}
        [\alpha, \beta]'_r :=~& [\alpha, \beta]' + \rho_{r(\alpha)} \beta - \rho_{r (\beta)} \alpha,\\
            \llbracket \alpha, \beta, \gamma \rrbracket'_r :=~& \llbracket \alpha, \beta, \gamma \rrbracket' + D_{\rho, \mu} ( r(\alpha), r (\beta)) \gamma + \mu ( r(\beta), r (\gamma)) \alpha - \mu ( r (\alpha), r (\gamma)) \beta, \text{ for } \alpha, \beta, \gamma \in \mathfrak{h}.
        \end{align*} 

\noindent  (This Lie-Yamaguti algebra $(\mathfrak{h}, [~,~]'_r, \llbracket ~, ~, ~ \rrbracket'_r)$ is simply denoted by $\mathfrak{h}_r$, and it is said to be induced by the deformation map $r$.)

        (ii) We define linear maps $\psi_r : \mathfrak{h} \rightarrow \mathrm{End}(\mathfrak{g})$ and $\nu_r : \mathfrak{h} \otimes \mathfrak{h} \rightarrow \mathrm{End}(\mathfrak{g})$ by 
     \begin{align*}
        ( \psi_r)_\alpha x :=~& \psi_\alpha x + [r(\alpha), x ] + r (\rho_x \alpha),\\
         (\nu_r )(\alpha, \beta) x :=~& \nu (\alpha, \beta) x + \llbracket x, r(\alpha), r (\beta) \rrbracket - r \big(  D_{\rho, \mu} (x, r(\alpha)) \beta - \mu (x, r (\beta))\alpha  \big),
     \end{align*}
     for all $\alpha, \beta \in \mathfrak{h}$ and $x \in \mathfrak{g}$. Then the triple $(\mathfrak{g}; \psi_r , \nu_r)$ is a representation of the induced Lie-Yamaguti algebra $\mathfrak{h}_r = (\mathfrak{h}, [~,~]'_r, \llbracket ~, ~, ~ \rrbracket'_r )$.
\end{thm}


\begin{proof} (i) This part follows from Proposition \ref{prop-graph} as $Graph (r)$ is linearly isomorphic to the space $\mathfrak{h}$ via the identification $(r(\alpha), \alpha) \leftrightarrow \alpha$, for $\alpha \in \mathfrak{h}$.

(ii) It follows from a direct calculation that the maps $\psi_r$ and $\nu_r$ satisfy the five conditions of a representation (a similar calculation is given in \cite{zhao-qiao,zhao-xu-qiao} when $r$ is a relative Rota-Baxter operator, and in the reference \cite{sun-chen} when $r$ is a derivation or a crossed homomorphism).
\end{proof}

\begin{remark}
    Note that, associated to the above representation $(\mathfrak{g}; \psi_r, \nu_r)$ of the induced Lie-Yamaguti algebra $\mathfrak{h}_r$, the corresponding map $D_{\psi_r, \nu_r} : \wedge^2 \mathfrak{h} \rightarrow \mathrm{End} (\mathfrak{g})$ is given by
    \begin{align*}
      D_{\psi_r, \nu_r} (\alpha, \beta) x= D_{\psi, \nu} (\alpha, \beta) x + \llbracket r(\alpha), r (\beta), x \rrbracket  + r \big(  \mu ( r (\alpha), x) \beta - \mu ( r (\beta), x) \alpha   \big), \text{ for } \alpha, \beta \in \mathfrak{h} \text{ and } x \in \mathfrak{g}.
    \end{align*}
\end{remark}

\medskip

In the following, we focus on the classifying complements problem in the context of Lie-Yamaguti algebras. Let $(E, [~, ~]_E, \llbracket ~, ~, ~ \rrbracket_E)$ be a Lie-Yamaguti algebra and $\mathfrak{g} \subset E$ be a Lie-Yamaguti subalgebra of it. Then a $\mathfrak{g}$-{\bf complement} in $E$ is a Lie-Yamaguti subalgebra $\mathfrak{h} \subset E$ such that $E= \mathfrak{g} + \mathfrak{h}$ and $\mathfrak{g} \cap \mathfrak{h} = \{ 0 \}$. It is said to be a {\bf strong} $\mathfrak{g}$-{\bf complement} if additionally $\llbracket \mathfrak{g}, \mathfrak{h}, \mathfrak{h} \rrbracket_E \subset \mathfrak{g}$ and $ \llbracket \mathfrak{h}, \mathfrak{g}, \mathfrak{g} \rrbracket_E \subset \mathfrak{h}$. In view of Theorem \ref{thm-factorization}, we have the following result.

\begin{prop}\label{prop-ly1}
    Let $\mathfrak{g}$ be a Lie-Yamaguti subalgebra of $E$ and let $\mathfrak{h}$ be any strong $\mathfrak{g}$-complement in $E$. Then there exists a matched pair of Lie-Yamaguti algebras $(\mathfrak{g}, \mathfrak{h}, (\rho, \mu), (\psi, \nu))$ such that $E \cong \mathfrak{g} \Join \mathfrak{h}$ as Lie-Yamaguti algebras.
\end{prop}

Let $\mathfrak{g}$ be a Lie-Yamaguti subalgebra of $E$ and $\mathfrak{h}$ be any strong $\mathfrak{g}$-complement in $E$. Then the matched pair $(\mathfrak{g}, \mathfrak{h}, (\rho, \mu), (\psi, \nu))$ constructed in the above proposition is called the {\em canonical matched pair}.

\begin{thm}\label{thm-ccp}
    Let $\mathfrak{g}$ be a Lie-Yamaguti subalgebra of $E$, and $\mathfrak{h}$ be any strong $\mathfrak{g}$-complement in $E$ with the associated canonical matched pair $(\mathfrak{g}, \mathfrak{h}, (\rho, \mu), (\psi, \nu))$. 

    (i) Then for any deformation map $r : \mathfrak{h} \rightarrow \mathfrak{g}$, the graph $Graph (r) = \{ (r (\alpha), \alpha) ~\! | ~ \! \alpha \in \mathfrak{h} \}$ is also a $\mathfrak{g}$-complement in $E$.

    (ii) Conversely, let $\overline{\mathfrak{h}}$ be any other $\mathfrak{g}$-complement in $E$. Then there exists a deformation map $r: \mathfrak{h} \rightarrow \mathfrak{g}$ in the canonical matched pair such that $\overline{\mathfrak{h}} = Graph (r)$ as Lie-Yamaguti algebras.
\end{thm}

\begin{proof}
(i) It has been shown in Proposition \ref{prop-graph} that $Graph (r)$ is a Lie-Yamaguti subalgebra of $\mathfrak{g} \Join \mathfrak{h} \cong E$. Moreover, it is easy to see that $E = \mathfrak{g} + Graph (r)$ and $\mathfrak{g} \cap Graph (r) = \{ 0\}$. Hence, $Graph (r)$ is a $\mathfrak{g}$-complement in $E$.

    (ii) Since $\overline{\mathfrak{h}}$ is a $\mathfrak{g}$-complement in $E$, we have $E = \mathfrak{g} \oplus \mathfrak{h} = \mathfrak{g} \oplus \overline{\mathfrak{h}}$. Hence, there exist linear maps
    \begin{align*}
        a : \mathfrak{h} \rightarrow \mathfrak{g}, \quad b : \mathfrak{h} \rightarrow \overline{\mathfrak{h}}, \quad c : \overline{\mathfrak{h}} \rightarrow \mathfrak{g} ~~~~ \text{ and } ~~~~ d : \overline{\mathfrak{h}} \rightarrow \mathfrak{h}
    \end{align*}
    such that for any $\alpha \in \mathfrak{h}$ and $s \in \overline{\mathfrak{h}}$,
    \begin{align*}
        \alpha = a (\alpha) \oplus b (\alpha) ~~~~ \text{ and } ~~~~ s = c(s) \oplus d (s).
    \end{align*}
    For any $\alpha \in \mathfrak{h}$, the element $b (\alpha) \in \overline{ \mathfrak{h}}$ can be written as
    \begin{align*}
        - a (\alpha) \oplus \alpha = b (\alpha) = c (b (\alpha)) +  d(b (\alpha)).
    \end{align*}
    Hence, by the uniqueness of decomposition in a direct sum, we get that $d ( b (\alpha)) = \alpha$, for all $\alpha \in \mathfrak{h}$. In a similar way, for any $\overline{\alpha} \in \overline{\mathfrak{h}}$, the element $d (\overline{\alpha}) \in \mathfrak{h}$ can be written as
       $ - c (\overline{\alpha}) \oplus \overline{\alpha} = d (\overline{\alpha}) = a (d (\overline{\alpha})) + b (d (\overline{\alpha}))$, which implies that $b (d (\overline{\alpha})) = \overline{\alpha}$, for all $\overline{\alpha} \in \overline{\mathfrak{h}}$. This shows that the map $b: \mathfrak{h} \rightarrow \overline{\mathfrak{h}}$ is a linear isomorphism. We now claim that $r = -a$ is a deformation map in the canonical matched pair. First, for any $\alpha , \beta \in \mathfrak{h}$, we have from (\ref{bicrossed-1}) that
       \begin{align}\label{comparing-last}
            [  (r(\alpha), \alpha) , ( r (\beta), \beta)]_\Join = \big(  [r(\alpha), r(\beta)] + \psi_\alpha r (\beta) - \psi_\beta r (\alpha) ~ \!, ~ \! [\alpha, \beta]' + \rho_{r (\alpha)} \beta - \rho_{r (\beta)} \alpha   \big).
       \end{align}
  On the other hand, we see that
  \begin{align*}
      [ (r(\alpha), \alpha) , ( r (\beta), \beta)]_\Join = [ b(\alpha), b (\beta)]_\Join = b (\theta) = (r (\theta), \theta),
  \end{align*}
  for some unique $\theta \in \mathfrak{h}$. The second equality of the above identity holds as $b (\alpha), b (\beta) \in \overline{\mathfrak{h}}$ and $\overline{\mathfrak{h}}$ is a Lie-Yamaguti subalgebra of $E \cong \mathfrak{g} \Join \mathfrak{h}$. Comparing the last expression with that of (\ref{comparing-last}), we get that the map $r$ satisfies the identity (\ref{defor-1}). Similarly, for $\alpha, \beta, \gamma \in \mathfrak{h}$, we have from (\ref{bicrossed-2}) that
  \begin{align*}
      \llbracket (r (\alpha), \alpha), (r (\beta), \beta) , (r (\gamma), \gamma) \rrbracket_\Join =~& \big( \llbracket r (\alpha),  r (\beta), r (\gamma) \rrbracket+ D_{\psi, \nu} (\alpha, \beta) r (\gamma) + \nu (\beta, \gamma) r (\alpha) - \nu (\alpha, \gamma) r (\beta)  ~ \!, \nonumber \\
         & \quad \llbracket \alpha, \beta, \gamma \rrbracket' + D_{\rho, \mu} ( r (\alpha), r (\beta)) \gamma + \mu ( r (\beta),  r (\gamma)) \alpha - \mu ( r (\alpha), r (\gamma)) \beta \big).
  \end{align*}
  On the other hand,
  \begin{align*}
      \llbracket (r (\alpha), \alpha), (r (\beta), \beta) , (r (\gamma), \gamma) \rrbracket_\Join = \llbracket b (\alpha), b (\beta), b (\gamma) \rrbracket_\Join = b (\theta')= (r (\theta'), \theta'),
  \end{align*}
  for some unique $\theta' \in \mathfrak{h}$. Comparing this with the above identity, we get that $r$ also satisfies the identity (\ref{defor-2}). Hence, our claim follows. Finally, 
  \begin{align*}
      Graph (r) = \{ ( - a (\alpha), \alpha ) ~\! | ~\! \alpha \in \mathfrak{h} \} = \{ b (\alpha) ~ \! | ~ \! \alpha \in \mathfrak{h} \} = \overline{\mathfrak{h}}.
  \end{align*}
  This completes the proof.
\end{proof}

\begin{defn}\label{defn-equiv}
    Let $(\mathfrak{g}, \mathfrak{h}, (\rho, \mu), (\psi, \nu))$ be a matched pair of Lie-Yamaguti algebras. Two deformation maps $r, r' :\mathfrak{h} \rightarrow \mathfrak{g}$ are said to be {\bf equivalent} (and we write $r \sim r'$) if there exists a linear isomorphism $\sigma : \mathfrak{h} \rightarrow \mathfrak{h}$ that satisfies
    \begin{align*}
        \sigma ([\alpha, \beta]') - [\sigma (\alpha), \sigma (\beta)]' = \rho_{r' (\sigma (\alpha))} \sigma (\beta) - \rho_{r' (\sigma (\beta))} \sigma (\alpha) - \sigma ( \rho_{r (\alpha)} \beta -  \rho_{r (\beta)} \alpha  ),
        \end{align*}
        \begin{align*}
        \sigma (\llbracket \alpha, \beta, \gamma \rrbracket' ) ~&- \llbracket \sigma (\alpha), \sigma (\beta), \sigma (\gamma) \rrbracket' = D_{\rho, \mu} \big( r' (\sigma (\alpha)), r' (\sigma (\beta))  \big) \sigma (\gamma) + \mu \big(   r' (\sigma (\beta)),  r' (\sigma (\gamma))     \big) \sigma (\alpha) \\
        &- \mu \big(   r' (\sigma (\alpha)),  r' (\sigma (\gamma))     \big) \sigma (\beta) - \sigma \big(   D_{\rho, \mu} ( r(\alpha), r (\beta)) \gamma + \mu ( r(\beta), r (\gamma)) \alpha - \mu ( r (\alpha), r (\gamma)) \beta    \big),
    \end{align*}
    for all $\alpha, \beta, \gamma \in \mathfrak{h}$.
\end{defn}

We are now ready to prove the main result of this section that classifies the set of all isomorphism classes of complements.


\begin{thm}\label{thm-ccp2}
Let $\mathfrak{g}$ be a Lie-Yamaguti subalgebra of $E$, and $\mathfrak{h}$ be any strong $\mathfrak{g}$-complement in $E$ with the associated canonical matched pair $(\mathfrak{g}, \mathfrak{h},  (\rho, \mu), (\psi, \nu))$.
\begin{itemize}
    \item[(i)] Then $\sim$ is an equivalence relation on the set $\mathcal{DM} (\mathfrak{g}, \mathfrak{h})$.
    \item[(ii)] Moreover, there is a bijection between the isomorphism classes of all $\mathfrak{g}$-complements in $E$ and the space $ \mathcal{DM} (\mathfrak{g}, \mathfrak{h}) / \sim $. In particular, the factorization index of $\mathfrak{g}$ in $E$ is given by
    \begin{align*}
        [E : \mathfrak{g} ]^f = |   \mathcal{DM} (\mathfrak{g}, \mathfrak{h}) / \sim  |.
    \end{align*}
\end{itemize}
\end{thm}

\begin{proof}
    Two deformation maps $r$ and $r'$ are equivalent in the sense of Definition \ref{defn-equiv} if and only if the corresponding induced Lie-Yamaguti algebras $Graph (r)$ and $Graph (r')$ are isomorphic by the map $(r (\alpha), \alpha) \leftrightarrow ( r' (\sigma (\alpha)) , \sigma (\alpha))$. Hence, it follows that $\sim$ is an equivalence relation on the set $\mathcal{DM} (\mathfrak{g}, \mathfrak{h}).$

    By Theorem \ref{thm-ccp} (ii), any $\mathfrak{g}$-complement in $E$ is given by $Graph (r)$, for some deformation map $r$. Thus, by the previous paragraph, two $\mathfrak{g}$-complements are isomorphic if and only if the corresponding deformation maps are equivalent. Hence, the last part also follows. 
\end{proof}

\medskip

\subsection{Factorizations, deformation maps and the classifying complements problem for Lie triple systems}
As discussed earlier, any Lie triple system can be viewed as a Lie-Yamaguti algebra whose bracket $[~,~]$ is trivial. Hence, by making certain operations trivial in the previous discussions, we obtain the corresponding results for Lie triple systems.

Let $(\mathfrak{g}, \llbracket ~, ~, ~ \rrbracket)$ be a Lie triple system. First, recall that \cite{jacobson,lister,yamaguti0} (also follows from the representation of a Lie-Yamaguti algebra by taking $[~,~]=0$ and $\rho = 0$) a representation of  $(\mathfrak{g}, \llbracket ~, ~, ~ \rrbracket)$ is a pair $(V; \mu)$ consisting of a vector space $V$ endowed with a linear map $\mu: \mathfrak{g} \otimes \mathfrak{g} \rightarrow \mathrm{End} (V)$ that satisfy the following conditions:
\begin{equation*}
    \mu (z, w) \mu (x, y) - \mu (y, w) \mu (x, z) - \mu (x, \llbracket y, z, w \rrbracket) + D_\mu (y, z) \mu (x, w) = 0,
    \end{equation*}
    \begin{equation*}
       \mu ( \llbracket x, y, z \rrbracket, w) + \mu (z, \llbracket x, y, w \rrbracket) = D_{\mu} (x, y) \mu(z, w) - \mu (z, w) D_\mu (x, y),
\end{equation*}
for all $x, y, z , w \in \mathfrak{g}$. Here we have used the notation $D_\mu (x, y) := \mu (y, x) - \mu (x, y)$, for any $x, y \in \mathfrak{g}$.

\begin{defn}
    A {\bf matched pair of Lie triple systems} is a tuple $\big( (\mathfrak{g}, \llbracket ~, ~, ~ \rrbracket), (\mathfrak{h}, \llbracket ~, ~, ~ \rrbracket'), \mu, \nu \big)$ consisting of two Lie triple systems $(\mathfrak{g}, \llbracket ~, ~, ~ \rrbracket)$ and $(\mathfrak{h}, \llbracket ~, ~, ~ \rrbracket')$ with two linear maps $\mu : \mathfrak{g} \otimes \mathfrak{g} \rightarrow \mathrm{End} (\mathfrak{h})$ and  $\nu : \mathfrak{h} \otimes \mathfrak{h} \rightarrow \mathrm{End} (\mathfrak{g})$ such that (i) $(\mathfrak{h}; \mu)$ is a representation of the Lie triple system $(\mathfrak{g}, \llbracket ~, ~, ~ \rrbracket)$, (ii) $(\mathfrak{g}; \nu)$ is a representation of the Lie triple system $(\mathfrak{h}, \llbracket ~, ~, ~ \rrbracket')$, and for all $x, y, z \in \mathfrak{g}$, $\alpha, \beta, \gamma \in \mathfrak{h}$, the following conditions hold:
     \begin{equation*}
        \mu (x, y) \llbracket \alpha, \beta, \gamma \rrbracket' = \llbracket \alpha, \beta, \mu (x, y) \gamma \rrbracket' - \mu ( D_{\nu} (\alpha, \beta) x, y) \gamma - \mu (x, D_{ \nu} (\alpha, \beta)y) \gamma,
    \end{equation*}
      \begin{equation*}
        \mu (\nu (\alpha, \beta) x, y) \gamma - \mu (\nu (\alpha, \gamma) x, y) \beta = \mu (x, D_{ \nu} (\beta, \gamma) y) \alpha - \llbracket \beta, \gamma, \mu (x, y) \alpha \rrbracket',
    \end{equation*}
    \begin{equation*}
        \mu (x, \nu (\alpha, \beta) y) \gamma = \llbracket \mu (x, y) \gamma, \alpha, \beta \rrbracket' - D_{\mu } (y, \nu (\gamma, \alpha) x ) \beta + \mu ( y, \nu (\gamma, \beta) x) \alpha,
    \end{equation*}
    \begin{equation*}
        \nu (\alpha, \beta) \llbracket x, y, z \rrbracket = \llbracket x, y, \nu (\alpha, \beta) z \rrbracket - \nu ( D_{ \mu} (x, y) \alpha , \beta) z - \nu (\alpha , D_{\mu} (x, y) \beta) z,
    \end{equation*}
    \begin{equation*}
        \nu (\mu (x, y) \alpha, \beta) z - \nu ( \mu (x, z) \alpha , \beta) y = \nu (\alpha, D_{ \mu} (y, z) \beta) x - \llbracket y, z, \nu (\alpha, \beta) x \rrbracket,
    \end{equation*}
    \begin{equation*}
        \nu (\alpha, \mu (x, y) \beta) z = \llbracket \nu (\alpha , \beta) z, x, y \rrbracket - D_{ \nu} (\beta, \mu (z, x ) \alpha ) y + \nu (\beta, \mu (z, y) \alpha) x.
    \end{equation*}
\end{defn}

\medskip

Let $\big( (\mathfrak{g}, \llbracket ~, ~, ~ \rrbracket), (\mathfrak{h}, \llbracket ~, ~, ~ \rrbracket'), \mu, \nu \big)$  be any matched pair of Lie triple systems. Then $(\mathfrak{g} \oplus \mathfrak{h}, \llbracket ~, ~, ~ \rrbracket_\Join)$ is a Lie triple system, where
\begin{align*}
     \llbracket (x, \alpha), (y, \beta), (z, \gamma) \rrbracket_\Join :=~& \big( \llbracket x, y, z \rrbracket+ D_{ \nu} (\alpha, \beta) z + \nu (\beta, \gamma) x - \nu (\alpha, \gamma) y  ~ \!, \nonumber \\
         & \qquad \llbracket \alpha, \beta, \gamma \rrbracket' + D_{ \mu} (x, y) \gamma + \mu (y, z) \alpha - \mu (x, z) \beta \big),
\end{align*}
for $(x, \alpha), (y, \beta), (z, \gamma) \in \mathfrak{g} \oplus \mathfrak{h}$. This is called the {\em bicrossed product} of the given matched pair of Lie triple systems. It is simply denoted by $\mathfrak{g} \Join \mathfrak{h}$.

\medskip

Let $(\mathfrak{g}, \llbracket ~, ~, ~ \rrbracket)$ and $(\mathfrak{h}, \llbracket ~, ~, ~ \rrbracket')$ be two Lie triple systems. Then a Lie triple system $(E, \llbracket ~, ~, ~ \rrbracket_E)$ is said to be strongly factorize through $\mathfrak{g}$ and $\mathfrak{h}$ if
\begin{itemize}
    \item[(i)] $\mathfrak{g}$ and $\mathfrak{h}$ are both subalgebras of the Lie triple system $(E, \llbracket ~, ~, ~ \rrbracket_E)$,
    \item[(ii)] $E = \mathfrak{g} + \mathfrak{h}$ and $\mathfrak{g} \cap \mathfrak{h} = \{ 0\},$
    \item[(iii)] $\llbracket \mathfrak{g}, \mathfrak{h}, \mathfrak{h} \rrbracket_E \subset \mathfrak{g}$ and $\llbracket \mathfrak{h}, \mathfrak{g}, \mathfrak{g} \rrbracket_E \subset \mathfrak{h}$.
\end{itemize}
In view of Theorem \ref{thm-factorization}, we have the following result.

\begin{thm}
    Let $(\mathfrak{g}, \llbracket ~, ~, ~ \rrbracket)$ and $(\mathfrak{h}, \llbracket ~, ~, ~ \rrbracket')$ be two Lie triple systems. Then a Lie triple system $(E, \llbracket ~, ~, ~ \rrbracket_E)$ strongly factorizes through $\mathfrak{g}$ and $\mathfrak{h}$ if and only if there exists a matched pair of Lie triple systems $( (\mathfrak{g}, \llbracket ~, ~, ~ \rrbracket), (\mathfrak{h}, \llbracket ~, ~, ~ \rrbracket'), \mu, \nu )$ such that $E \cong \mathfrak{g} \Join \mathfrak{h}$ as Lie triple systems.
\end{thm}

\begin{defn}
    Let $\big( (\mathfrak{g}, \llbracket ~, ~, ~ \rrbracket), (\mathfrak{h}, \llbracket ~, ~, ~ \rrbracket'), \mu, \nu \big)$ be a matched pair of Lie triple systems. Then a linear map $r : \mathfrak{h} \rightarrow \mathfrak{g}$ is said to be a {\bf deformation map} if for all $\alpha, \beta, \gamma \in \mathfrak{h}$,
    \begin{align*}
        \llbracket r (\alpha), r (\beta), r (\gamma) \rrbracket + D_{\nu} (\alpha, \beta) r(\gamma) & + \nu (\beta, \gamma) r (\alpha) - \nu (\alpha , \gamma) r (\beta) \nonumber \\
     = r \big(  \llbracket \alpha , \beta, \gamma \rrbracket' +& D_{\mu} (r (\alpha), r (\beta)) \gamma + \mu ( r (\beta), r (\gamma)) \alpha - \mu (r (\alpha), r (\gamma)) \beta \big).
    \end{align*}
\end{defn}

It follows that a linear map $r : \mathfrak{h} \rightarrow \mathfrak{g}$ is a deformation map if and only if $Graph (r) = \{    (r(\alpha), \alpha) ~ \! | ~ \! \alpha \in \mathfrak{h} \}$ is a subalgebra of the bicrossed product Lie triple system $\mathfrak{g} \Join \mathfrak{h} = (\mathfrak{g} \oplus \mathfrak{h}, \llbracket ~, ~ , ~ \rrbracket_\Join)$.

\begin{thm}
    Let $\big( (\mathfrak{g}, \llbracket ~, ~, ~ \rrbracket), (\mathfrak{h}, \llbracket ~, ~, ~ \rrbracket'), \mu, \nu \big)$ be a matched pair of Lie triple systems and $r : \mathfrak{h} \rightarrow \mathfrak{g}$ be a deformation map.
   
    (i) Then the vector space $\mathfrak{h}$ inherits a new Lie triple system structure with the operation
        \begin{align*}
             \llbracket \alpha, \beta, \gamma \rrbracket'_r := \llbracket \alpha, \beta, \gamma \rrbracket' + D_{ \mu} ( r(\alpha), r (\beta)) \gamma + \mu ( r(\beta), r (\gamma)) \alpha - \mu ( r (\alpha), r (\gamma)) \beta, \text{ for } \alpha, \beta, \gamma \in \mathfrak{h}.
        \end{align*}
        
    (ii) We also define a linear map $\nu_r : \mathfrak{h} \otimes \mathfrak{h} \rightarrow \mathrm{End} (\mathfrak{g})$ by
        \begin{align*}
             (\nu_r )(\alpha, \beta) x :=~& \nu (\alpha, \beta) x + \llbracket x, r(\alpha), r (\beta) \rrbracket - r \big(  D_{ \mu} (x, r(\alpha)) \beta - \mu (x, r (\beta))\alpha  \big),
        \end{align*}
        for $\alpha, \beta \in \mathfrak{h}$ and $x \in \mathfrak{g}$.
        Then $(\mathfrak{g}; \nu_r)$ is a representation of the induced Lie triple system $(\mathfrak{h}, \llbracket ~, ~, ~ \rrbracket'_r)$.
\end{thm}

Let $(E, \llbracket ~, ~, ~ \rrbracket_E)$ be a Lie triple system and $\mathfrak{g} \subset E$ be a subalgebra of it. Then a $\mathfrak{g}$-complement (resp. strong $\mathfrak{g}$-complement) in $E$ is a subalgebra $\mathfrak{h} \subset E$ such that
\begin{align*}
    E= \mathfrak{g} + \mathfrak{h}, \quad \mathfrak{g} \cap \mathfrak{h} = \{ 0 \} \quad (\text{additionally, } \llbracket \mathfrak{g}, \mathfrak{h}, \mathfrak{h} \rrbracket_E \subset \mathfrak{g} ~~~ \text{ and } ~~~ \llbracket \mathfrak{h}, \mathfrak{g}, \mathfrak{g} \rrbracket_E \subset \mathfrak{h}).
\end{align*}
In the context of Lie triple systems, Proposition \ref{prop-ly1} and Theorem \ref{thm-ccp} can be summarized in the following result.

\begin{thm}
    Let $\mathfrak{g}$ be a subalgebra of the Lie triple system $E$, and let $\mathfrak{h}$ be any strong $\mathfrak{g}$-complement in $E$.

    (i) Then there exists a matched pair of Lie triple systems $(\mathfrak{g}, \mathfrak{h}, \mu, \nu)$, called the canonical matched pair, such that $E \cong \mathfrak{g} \Join \mathfrak{h}$ as Lie triple systems.

    (ii) For any deformation map $r : \mathfrak{h} \rightarrow \mathfrak{g}$ in the canonical matched pair, the graph $Graph (r)$ is a $\mathfrak{g}$-complement in $E$.

    (iii) Finally, if $\overline{\mathfrak{h}}$ is any $\mathfrak{g}$-complement in $E$, there exists a deformation map $r : \mathfrak{h} \rightarrow \mathfrak{g}$ in the canonical matched pair such that $\overline{ \mathfrak{h}} = Graph (r)$ as Lie triple systems.
\end{thm}

One may also define the equivalence relation $~\sim~$ (similar to Definition \ref{defn-equiv}) on the set of all deformation maps in a given matched pair of Lie triple systems. Then there is a bijection between the isomorphism classes of all $\mathfrak{g}$-complements in $E$ and the equivalence classes of all deformation maps in the canonical matched pair.

\section{Cohomology and the governing $L_\infty$-algebra of a deformation map}\label{section-5}
In this section, we first define the cohomology of a deformation map $r$ in a given matched pair of Lie-Yamaguti algebras. Our cohomology unifies the existing cohomologies of crossed homomorphisms and Rota-Baxter operators on Lie-Yamaguti algebras. Given a matched pair, we also construct an $L_\infty$-algebra whose Maurer-Cartan elements are precisely deformation maps. Finally, using this Maurer-Cartan characterization, we obtain the governing $L_\infty$-algebra of a deformation map $r$.

\subsection{Cohomology of deformation maps}\label{subsection-51}

Let $(\mathfrak{g}, \mathfrak{h}, (\rho, \mu), (\psi, \nu))$ be a matched pair of Lie-Yamaguti algebras and $r : \mathfrak{h} \rightarrow \mathfrak{g}$ be a deformation map. In Theorem \ref{thm-induced}, we have seen that $r$ induces a new Lie-Yamaguti algebra structure $(\mathfrak{h} , [~,~]'_r, \llbracket ~, ~, ~ \rrbracket_r')$ on the vector space $\mathfrak{h}$. Moreover, this induced Lie-Yamaguti algebra has a representation $(\mathfrak{g}; \psi_r,\nu_r)$. Therefore, one may consider the cochain complex $\{ C^\bullet (\mathfrak{h}, \mathfrak{g}), \delta_r  \}$ of the Lie-Yamaguti algebra $(\mathfrak{h} , [~,~]'_r, \llbracket ~, ~, ~ \rrbracket_r')$ with coefficients in the above representation. Note that the coboundary map has been denoted by the notation $\delta_r$ to emphasize that it is induced by the deformation map $r$.

We first observe the following result.

\begin{prop}\label{prop-1co}
    For any $\mathfrak{X} = x \wedge y \in \wedge^2 \mathfrak{g}$, we define a linear map $d (\mathfrak{X}) : \mathfrak{h} \rightarrow \mathfrak{g}$ by
\begin{align*}
    d (\mathfrak{X}) \alpha = d (x \wedge y) \alpha := r (D_{\rho, \mu} (x, y) \alpha) - \llbracket x, y, r (\alpha) \rrbracket, \text{ for all } \alpha \in \mathfrak{h}. 
\end{align*}
Then  $d (\mathfrak{X}) \in \mathrm{Hom} (\mathfrak{h}, \mathfrak{g}) = C^1 (\mathfrak{h}, \mathfrak{g})$ is a $1$-cocycle in the cochain complex $\{ C^\bullet (\mathfrak{h}, \mathfrak{g}), \delta_r \}$.
\end{prop}

\begin{proof}
    To show that $d (\mathfrak{X})$ is a $1$-cocycle, we need to show that $\delta_r ( d (\mathfrak{X})) = \big(  \delta_r ( d (\mathfrak{X}))_\mathrm{I}, \delta_r ( d (\mathfrak{X}))_\mathrm{II} \big) = 0$. For any $\alpha, \beta \in \mathfrak{h}$, we observe that
    \begin{align*}
        & \delta_r ( d(\mathfrak{X} ))_\mathrm{I} (\alpha, \beta)\\
        &= (\psi_r)_\alpha \delta (\mathfrak{X}) \beta -  (\psi_r)_\beta \delta (\mathfrak{X}) \alpha - \delta (\mathfrak{X}) ([\alpha, \beta]'_r) \\
        &= \psi_\alpha \delta (\mathfrak{X}) \beta + [r(\alpha), \delta (\mathfrak{X}) \beta] + r (\rho_{\delta (\mathfrak{X}) \beta} \alpha) - \psi_\beta \delta (\mathfrak{X}) \alpha - [r(\beta), \delta (\mathfrak{X}) \alpha] - r (\rho_{\delta (\mathfrak{X}) \alpha} \beta) \\
        & \quad - r (D_{\rho, \mu} (x, y) [\alpha, \beta]'_r) + \llbracket x, y, r ([\alpha, \beta]'_r) \rrbracket \\
        &= \psi_\alpha r (D_{\rho, \mu} (x, y) \beta) - \psi_\alpha \llbracket x, y, r (\beta) \rrbracket + [r (\alpha), r (D_{\rho, \mu} (x, y) \beta)] - [r (\alpha), \llbracket x, y, r (\beta) \rrbracket ] \\
        & \quad + r \big(  \rho_{ r (D_{\rho, \mu} (x, y) \beta ) } \alpha - \rho_{\llbracket x, y, r (\beta) \rrbracket } \alpha    \big) - \psi_\beta r (D_{\mu, \rho} (x, y) \alpha ) + \psi_\beta \llbracket x, y, r (\alpha) \rrbracket \\
        & \quad - [ r(\beta), r (D_{\rho, \mu} (x, y) \alpha )] + [ r (\beta) , \llbracket x, y, r (\alpha) \rrbracket ] - r \big(  \rho_{ r (D_{\rho, \mu} (x, y) \alpha ) } \beta - \rho_{\llbracket x, y, r (\alpha) \rrbracket } \beta  \big) \\
        & \quad - r \big(  [D_{\rho, \mu} (x, y) \alpha, \beta]' + [\alpha, D_{\rho, \mu} (x, y) \beta]' + D_{\rho, \mu} (x, y) \rho_{r (\alpha)} \beta - D_{\rho, \mu} (x, y) \rho_{r (\beta)} \alpha   \big) \\
        & \quad + \llbracket x, y, [r (\alpha), r (\beta)] \rrbracket + \llbracket x, y, \psi_\alpha r (\beta) \rrbracket - \llbracket x, y, \psi_\beta r(\alpha) \rrbracket \\
        &= 0 \quad (\text{by using } (\ref{eqn-3}),(\ref{eqn-16}), (\ref{defor-1})).
    \end{align*}
    In the same way, one deduce that $\delta_r  ( d(\mathfrak{X} ))_\mathrm{II} (\alpha, \beta, \gamma) = 0$, for all $\alpha, \beta, \gamma = 0$. This proves the result.
\end{proof}

We are now ready to define the cohomology of a deformation map $r$. For each $n \geq 0$, we define the space $C^n_r (\mathfrak{h}, \mathfrak{g})$ of $n$-cochains by
\begin{align*}
    C^n_r (\mathfrak{h}, \mathfrak{g}) := \begin{cases}\wedge^2 \mathfrak{g} & \text{ if } n = 0,\\
        C^n (\mathfrak{h}, \mathfrak{g}) & \text{ if } n \geq 1.
    \end{cases}
\end{align*}
A map $\partial_r : C^n_r (\mathfrak{h}, \mathfrak{g}) \rightarrow C^{n+1}_r (\mathfrak{h}, \mathfrak{g}) $ is defined by 
\begin{align*}
    \partial_r (\mathfrak{X}) =~& d(\mathfrak{X}) ~~ \text{ if } \mathfrak{X} \in \wedge^2 \mathfrak{g} = C^0_r (\mathfrak{h}, \mathfrak{g}), \\
    \partial_r (f) =~& \delta_r (f) ~~ \text{ if } f \in C^{n \geq 1}_r (\mathfrak{h}, \mathfrak{g}).
\end{align*}
It turns out from Proposition \ref{prop-1co} that $\{ C^\bullet_r (\mathfrak{h}, \mathfrak{g}), \partial_r \}$ is a cochain complex. The corresponding cohomology groups are said to be the {\bf cohomology groups of the deformation map} $r$. They are denoted by $H^n_r (\mathfrak{h}, \mathfrak{g})$, for $n \geq 0$.

\subsection{From Lie-Yamaguti algebra to Maurer-Cartan element} Let $\mathfrak{g}$ be any vector space (need not have any additional structure). For any $p \geq 0$, we set
\begin{align*}
    C^{p+1} (\mathfrak{g}, \mathfrak{g}) := \begin{cases}
    \mathrm{Hom} (\mathfrak{g}, \mathfrak{g}) & \text{ if } p = 0,\\
        \mathrm{Hom} \big( \underbrace{(\wedge^2 \mathfrak{g}) \otimes \cdots \otimes (\wedge^2 \mathfrak{g})}_{p \text{~copies}} , \mathfrak{g} \big) \oplus \mathrm{Hom} \big(  \underbrace{(\wedge^2 \mathfrak{g}) \otimes \cdots \otimes (\wedge^2 \mathfrak{g})}_{p \text{~copies}} \otimes~\! \mathfrak{g} , \mathfrak{g} \big) & \text{ if } p \geq 1.
    \end{cases}
\end{align*}
Given any $f , g \in C^1 (\mathfrak{g}, \mathfrak{g}) = \mathrm{Hom} (\mathfrak{g}, \mathfrak{g})$, we define $f \diamond g \in C^1 (\mathfrak{g}, \mathfrak{g})$ to be the usual composition of $f$ and $g$. When  $f \in C^1 (\mathfrak{g}, \mathfrak{g})$ and $P = (P_\mathrm{I}, P_\mathrm{II}) \in C^{p+1} (\mathfrak{g}, \mathfrak{g})$, the elements  $f \diamond P = ( (f \diamond P)_\mathrm{I}, (f \diamond P)_\mathrm{II}) \in C^{p+1} (\mathfrak{g}, \mathfrak{g})$ and $P \diamond f  = (  (P \diamond f)_\mathrm{I}, (P \diamond f)_\mathrm{II}) \in C^{p+1} (\mathfrak{g}, \mathfrak{g})$ are respectively defined by using
\begin{align}
    (f \diamond P)_\mathrm{I} (\mathfrak{X}_1, \ldots, \mathfrak{X}_p) =~& f \big(  P_\mathrm{I} (\mathfrak{X}_1, \ldots, \mathfrak{X}_p) \big), \label{insert1} \\
    (f \diamond P)_\mathrm{II} (\mathfrak{X}_1, \ldots, \mathfrak{X}_p, x) =~& f \big( P_\mathrm{II} ( \mathfrak{X}_1, \ldots, \mathfrak{X}_p, x)    \big),  \label{insert2}  \\
    (P \diamond f)_\mathrm{I} ( \mathfrak{X}_1, \ldots, \mathfrak{X}_p)  =~& \sum_{k=1}^p P_\mathrm{I} \big( \mathfrak{X}_1, \ldots, \mathfrak{X}_{k-1}, f (x_k) \wedge y_k + x_k \wedge f(y_k), \mathfrak{X}_{k+1}, \ldots, \mathfrak{X}_p \big),  \label{insert3} \\
    (P \diamond f)_\mathrm{II} ( \mathfrak{X}_1, \ldots, \mathfrak{X}_p, x) =~&  \sum_{k=1}^p P_\mathrm{II} \big( \mathfrak{X}_1, \ldots, \mathfrak{X}_{k-1}, f (x_k) \wedge y_k + x_k \wedge f(y_k), \mathfrak{X}_{k+1}, \ldots, \mathfrak{X}_p, x  \big) \nonumber \\
    &+ P_\mathrm{II} (\mathfrak{X}_1, \ldots, \mathfrak{X}_p, f(x)),  \label{insert4} 
\end{align}
for $\mathfrak{X}_1, \ldots, \mathfrak{X}_p \in \wedge^2 \mathfrak{g}$ and $x \in \mathfrak{g}$.
Finally, when $P = (P_\mathrm{I}, P_\mathrm{II}) \in C^{p+1} (\mathfrak{g}, \mathfrak{g})$ and $Q = (Q_\mathrm{I}, Q_\mathrm{II}) \in C^{q+1} (\mathfrak{g}, \mathfrak{g})$, we define a new element $P \diamond Q = (( P \diamond Q)_\mathrm{I}, (P \diamond Q)_\mathrm{II}) \in C^{p+q+1} (\mathfrak{g}, \mathfrak{g})$ by the following
\begin{align*}
    &(P \diamond Q)_\mathrm{I} (\mathfrak{X}_1, \ldots, \mathfrak{X}_{p+q}) \\
    &= \sum_{\substack{\sigma \in \mathbb{S}_{(p, q)} \\ \sigma (p+q) = p+q }} (-1)^{pq} (-1)^\sigma P_\mathrm{II} (\mathfrak{X}_{\sigma (1)}, \ldots, \mathfrak{X}_{\sigma (p)} , Q_\mathrm{I} (\mathfrak{X}_{\sigma (p+1)}  , \ldots, \mathfrak{X}_{\sigma (p+q)} ) ) \\
    & ~ + \sum_{k=1}^p (-1)^{(k-1) q} \sum_{\sigma \in \mathbb{S}_{(k-1, q)}} (-1)^\sigma P_\mathrm{I} \big(  \mathfrak{X}_{\sigma (1)}, \ldots, Q_\mathrm{II} (\mathfrak{X}_{\sigma (k)} , \ldots, \mathfrak{X}_{\sigma (k+q-1)}, x_{k+q}) \wedge y_{k+q}, \mathfrak{X}_{k+q+1}, \ldots, \mathfrak{X}_{p+q}  \big) \\
    & ~ + \sum_{k=1}^p (-1)^{(k-1) q} \sum_{\sigma \in \mathbb{S}_{(k-1, q)}} (-1)^\sigma P_\mathrm{I} \big(  \mathfrak{X}_{\sigma (1)}, \ldots, x_{q+k} \wedge Q_\mathrm{II} (\mathfrak{X}_{\sigma (k)} , \ldots, \mathfrak{X}_{\sigma (k+q-1)}, y_{k+q}) , \mathfrak{X}_{k+q+1}, \ldots, \mathfrak{X}_{p+q}  \big),
\end{align*}
\begin{align*}
    &(P \diamond Q)_\mathrm{II} (\mathfrak{X}_1, \ldots, \mathfrak{X}_{p+q}, x) \\ 
    &= \sum_{\sigma \in \mathbb{S}_{(p, q)}} (-1)^{pq} (-1)^\sigma P_\mathrm{II} (\mathfrak{X}_{\sigma (1)}, \ldots, \mathfrak{X}_{\sigma (p)} , Q_\mathrm{II} (\mathfrak{X}_{\sigma (p+1)}  , \ldots, \mathfrak{X}_{\sigma (p+q)}, x ) ) \\
    & ~ + \sum_{k=1}^p (-1)^{(k-1) q} \sum_{\sigma \in \mathbb{S}_{(k-1, q)}} (-1)^\sigma P_\mathrm{II} \big(  \mathfrak{X}_{\sigma (1)}, \ldots, Q_\mathrm{II} (\mathfrak{X}_{\sigma (k)} , \ldots, \mathfrak{X}_{\sigma (k+q-1)}, x_{k+q}) \wedge y_{k+q}, \mathfrak{X}_{k+q+1}, \ldots, \mathfrak{X}_{p+q}, x  \big) \\
    & ~ + \sum_{k=1}^p (-1)^{(k-1) q} \sum_{\sigma \in \mathbb{S}_{(k-1, q)}} (-1)^\sigma P_\mathrm{II} \big(  \mathfrak{X}_{\sigma (1)}, \ldots, x_{q+k} \wedge Q_\mathrm{II} (\mathfrak{X}_{\sigma (k)} , \ldots, \mathfrak{X}_{\sigma (k+q-1)}, y_{k+q}) , \mathfrak{X}_{k+q+1}, \ldots, \mathfrak{X}_{p+q} , x \big),
\end{align*}
for all $\mathfrak{X}_1, \ldots, \mathfrak{X}_{p+q} \in \wedge^2 \mathfrak{g}$ and $x \in \mathfrak{g}$. As a result, we obtain a bilinear operation 
\begin{equation*}
  \diamond : C^{p+1} (\mathfrak{g}, \mathfrak{g}) \times C^{q+1} (\mathfrak{g}, \mathfrak{g}) \rightarrow C^{p+q+1} (\mathfrak{g}, \mathfrak{g}), \text{ for any } p, q \geq 0.  
\end{equation*}

Then the following result has been proved in \cite{zhao-qiao}.

\begin{thm}
    Let $\mathfrak{g}$ be any vector space. 
    
    (i) Define a degree $0$ graded skew-symmetric bracket $[~,~]_\mathsf{LY} : C^{\bullet +1} (\mathfrak{g}, \mathfrak{g}) \times C^{\bullet +1} (\mathfrak{g}, \mathfrak{g}) \rightarrow C^{\bullet +1} (\mathfrak{g}, \mathfrak{g})$ by
    \begin{align*}
        [P, Q]_\mathsf{LY} := P \diamond Q - (-1)^{p q} ~ \! Q \diamond P,
    \end{align*}
     for $P \in C^{p+1} (\mathfrak{g}, \mathfrak{g})$ and $ Q \in C^{q+1} (\mathfrak{g}, \mathfrak{g})$. Then $ \big(  C^{\bullet +1} (\mathfrak{g}, \mathfrak{g}), [~,~]_\mathsf{LY} \big)$
    is a graded Lie algebra.

    (ii) Let $(\mathfrak{g}, [~,~], \llbracket ~, ~, ~ \rrbracket)$ be a Lie-Yamaguti algebra structure on the vector space $\mathfrak{g}$. We define an element $\pi = (\pi_\mathrm{I}, \pi_\mathrm{II}) \in C^2 (\mathfrak{g}, \mathfrak{g})$ by
    \begin{align*}
        \pi_\mathrm{I} (x, y) = [x, y] ~~~~ \text{ and } ~~~~ \pi_\mathrm{II} (x, y, z) = \llbracket x, y, z \rrbracket, \text{ for } x, y, z \in \mathfrak{g}.
    \end{align*}
    Then $\pi$ is a Maurer-Cartan element of the graded Lie algebra $ \big(  C^{\bullet +1} (\mathfrak{g}, \mathfrak{g}), [~,~]_\mathsf{LY} \big).$
\end{thm}

Let $(\mathfrak{g}, [~,~], \llbracket ~, ~, ~ \rrbracket)$ be any Lie-Yamaguti algebra with the corresponding Maurer-Cartan element $\pi = (\pi_\mathrm{I}, \pi_\mathrm{II}) \in C^2 (\mathfrak{g}, \mathfrak{g})$ in the graded Lie algebra $(C^{\bullet +1 } (\mathfrak{g}, \mathfrak{g}) , [~, ~]_\mathsf{LY})$. Then, one may define a cochain complex
\begin{align}\label{delta-pi}
    C^1 (\mathfrak{g}, \mathfrak{g}) \xrightarrow{\delta_\pi} C^2 (\mathfrak{g}, \mathfrak{g}) \xrightarrow{\delta_\pi}  \cdots \cdots \xrightarrow{\delta_\pi} C^n (\mathfrak{g}, \mathfrak{g}) \xrightarrow{\delta_\pi} C^{n+1} (\mathfrak{g}, \mathfrak{g}) \xrightarrow{\delta_\pi} \cdots,
\end{align}
where $\delta_\pi (f) = [\pi, f]_\mathsf{LY}$ and $\delta_\pi (P) = (-1)^{n-1} ~ \! [\pi, P ]_\mathsf{LY}$, for $f \in C^1 (\mathfrak{g}, \mathfrak{g})$ and $P \in C^{n \geq 2} (\mathfrak{g}, \mathfrak{g})$. The corresponding cohomology groups are denoted by $H^\bullet (\mathfrak{g}, \mathfrak{g})$, and are called the cohomology groups of the given Lie-Yamaguti algebra.

 Note that the cohomology of a Lie-Yamaguti algebra with coefficients in a representation was also defined in Section \ref{section-2}. When $(V; \rho, \mu) = (\mathfrak{g}; \rho_\mathrm{ad}, \mu_\mathrm{ad})$ is the adjoint representation, the differential $\delta$ considered there coincides with the map $\delta_\pi$ given in (\ref{delta-pi}). Hence, the corresponding cohomologies are the same.

\subsection{Maurer-Cartan characterization and the governing $L_\infty$-algebra of a deformation map}

\begin{defn}
    An {\bf $L_\infty$-algebra} is a pair $(\mathcal{L}, \{ l_k \}_{k=1}^\infty)$ of a graded vector space $\mathcal{L} = \oplus_{n \in \mathbb{Z}} \mathcal{L}_n$ equipped with a collection $\{ l_k: \mathcal{L}^{\otimes k} \rightarrow \mathcal{L} \}_{k=1}^\infty$ of degree $1$ graded symmetric linear maps satisfying
\begin{align*}
 \displaystyle\sum^{N}_{i=1}~\sum_{\sigma\in \mathbb{S}_{(i,N-i)}} \epsilon(\sigma)~ \! l_{N-i+1}\big(l_i(x_{\sigma(1)},\ldots,x_{\sigma(i)}),x_{\sigma(i+1)},\ldots,x_{\sigma(N)}\big) = 0, 
\end{align*}
for any $N \geq 1$ and homogeneous elements $x_1, \ldots, x_N \in \mathcal{L}$. Here $ \epsilon (\sigma) = \epsilon (\sigma; x_1, \ldots, x_N)$ is the {\em Koszul sign} that appears in the graded context.
\end{defn}

\begin{defn} Let $(\mathcal{L}, \{ l_k \}_{k=1}^\infty)$ be an $L_\infty$-algebra. An element $c \in \mathcal{L}_0$ is said to be a {\bf Maurer-Cartan element} of this $L_\infty$-algebra if $c$ satisfies
\begin{align}\label{mc-eqn}
     \displaystyle\sum_{k=1}^{\infty} \dfrac{1}{k!} ~ \! l_k(c, \ldots, c) = 0.
\end{align}
\end{defn}

Throughout this paper, we assume that all $L_\infty$-algebras under consideration are {\em weakly filtered} \cite{getzler} so that infinite sums like (\ref{mc-eqn}) are always convergent. A useful construction of an $L_\infty$-algebra (that is also weakly filtered in practical cases) is given in \cite{voro}. To describe this construction, first recall that a {\em $V$-data} is a quadruple $(\mathfrak{L}, \mathfrak{a}, p, \Delta)$ consisting of a graded Lie algebra $(\mathfrak{L}, [ ~, ~ ])$, an abelian graded Lie subalgebra $\mathfrak{a} \subset \mathfrak{L}$, a projection map $p : \mathfrak{L} \rightarrow \mathfrak{a} \subset \mathfrak{L}$ whose kernel $\mathrm{ker}(p) \subset \mathfrak{L}$ is a graded Lie subalgebra and an element $\Delta \in \mathrm{ker}(p)_1$ that satisfies $[ \Delta, \Delta ] = 0$.

\begin{thm}\label{thm-voro}
    Let  $(\mathfrak{L}, \mathfrak{a}, p, \Delta)$ be a $V$-data. For any $k \geq 1$, we define a degree $1$ graded linear map $l_k : \mathfrak{a}^{\otimes k} \rightarrow \mathfrak{a}$ by
      \begin{align*} 
  l_k(a_1,\ldots,a_k) := p [ \cdots [ [ \Delta,a_1 ] ,a_2 ] ,\ldots,a_k ], 
  \end{align*}
  for all homogeneous elements $a_1, \ldots, a_k \in \mathfrak{a}$. Then $(\mathfrak{a}, \{ l_k \}_{k=1}^\infty)$ is an $L_\infty$-algebra.
\end{thm}

\medskip

Let $(\mathfrak{g}, \mathfrak{h}, (\rho, \mu), (\psi, \nu))$ be a matched pair of Lie-Yamaguti algebras with the corresponding bicrossed product $(\mathfrak{g} \oplus \mathfrak{h}, [~,~]_\Join, \llbracket ~, ~, ~ \rrbracket_\Join)$. In the following, we construct an $L_\infty$-algebra whose Maurer-Cartan elements correspond precisely to deformation maps. For this, we first consider the graded Lie algebra $\mathfrak{L} = \big( C^{\bullet +1} (\mathfrak{g} \oplus \mathfrak{h}, \mathfrak{g} \oplus \mathfrak{h}), [~,~]_\mathsf{LY} \big)$ associated to the vector space $\mathfrak{g} \oplus \mathfrak{h}$. It is then easy to see that $\mathfrak{a} = C^{\bullet +1} (\mathfrak{h}, \mathfrak{g})$ is a graded Lie subalgebra of $\mathfrak{L}$. Let $p: \mathfrak{L} \rightarrow \mathfrak{a} \subset \mathfrak{L}$ be the projection onto the subspace $\mathfrak{a}$. Then $\mathrm{ker} (p) \subset \mathfrak{L}$ is a graded Lie subalgebra. Further, since  $(\mathfrak{g} \oplus \mathfrak{h}, [~,~]_\Join, \llbracket ~, ~, ~ \rrbracket_\Join)$ is a Lie-Yamaguti algebra structure on $\mathfrak{g} \oplus \mathfrak{h}$, it defines a Maurer-Cartan element $\Pi = (\Pi_\mathrm{I}, \Pi_\mathrm{II}) \in \mathfrak{L}_1 = C^2 (\mathfrak{g} \oplus \mathfrak{h}, \mathfrak{g} \oplus \mathfrak{h})$ in the graded Lie algebra $\mathfrak{L} = \big(  C^{\bullet +1} (\mathfrak{g} \oplus \mathfrak{h}, \mathfrak{g} \oplus \mathfrak{h}), [~, ~]_\mathsf{LY} \big)$. Explicitly, 
\begin{align*}
    \Pi_\mathrm{I} ((x, \alpha), (y, \beta) ) = [ (x, \alpha), (y, \beta)]_\Join ~~~~ \text{ and } ~~~~ \Pi_\mathrm{II} ((x, \alpha), (y, \beta), (z, \gamma) ) = \llbracket (x, \alpha), (y, \beta), (z, \gamma) \rrbracket_\Join,
\end{align*}
for $(x, \alpha), (y, \beta), (z, \gamma) \in \mathfrak{g} \oplus \mathfrak{h}$. We also note that $\Pi \in \mathrm{ker} (p)$. Hence, we obtain the following result.

\begin{thm}\label{thm-mc-char}
    Let $(\mathfrak{g}, \mathfrak{h}, (\rho, \mu), (\psi, \nu))$ be a matched pair of Lie-Yamaguti algebras.

    (i) With the above notations, $( \mathfrak{L}, \mathfrak{a}, p, \Pi)$ is a $V$-data.

    (ii) Then $(\mathfrak{a} = C^{\bullet +1} (\mathfrak{h}, \mathfrak{g}), \{ l_k \}_{k=1}^\infty)$ is an $L_\infty$-algebra, where for homogeneous elements $P, Q, R \in \mathfrak{a}$,
    \begin{align*}
      l_1 (P) =~& p [ \Pi, P]_\mathsf{LY}, \quad  l_2 (P, Q) = p [ [ \Pi, P]_\mathsf{LY}, Q]_\mathsf{LY}, \quad   l_3 (P, Q, R) = p [ [ [ \Pi, P]_\mathsf{LY}, Q]_\mathsf{LY}, R ]_\mathsf{LY} \\
     & \text{ and } ~ l_k = 0, \text{ for } k \geq 4.
    \end{align*}

    (iii) A linear map $r : \mathfrak{h} \rightarrow \mathfrak{g}$ (i.e., $r \in \mathrm{Hom} (\mathfrak{h}, \mathfrak{g}) = C^1 (\mathfrak{h}, \mathfrak{g}) = \mathfrak{a}_0$) is a deformation map in the given matched pair of Lie-Yamaguti algebras if and only if $r \in \mathfrak{a}_0$ is a Maurer-Cartan element of the $L_\infty$-algebra $(\mathfrak{a}, \{ l_k \}_{k=1}^\infty )$.
\end{thm}

\begin{proof}
    (i) Follows from the preceding discussion.

    (ii) When part (i) applies to Theorem \ref{thm-voro}, we get an $L_\infty$-algebra $(\mathfrak{a} = C^{\bullet +1} (\mathfrak{h}, \mathfrak{g}), \{ l_k \}_{k=1}^\infty)$. For $k \geq 4$, a direct computation (by applying inputs) shows that
    \begin{align*}
        l_k (P_1, \ldots, P_k) = p [ \cdots [ [ \Pi, P_1]_\mathsf{LY}, P_2]_\mathsf{LY}, \ldots, P_k ]_\mathsf{LY} = 0.
    \end{align*}
    Hence, the result follows.

    (iii) Let $r: \mathfrak{h} \rightarrow \mathfrak{g}$ be any linear map (viewed as an element $r \in \mathfrak{a}_0$). Then for any $\alpha, \beta, \gamma \in \mathfrak{h}$, we observe that
    \begin{align*}
        &(l_1 (r))_\mathsf{I} (\alpha, \beta) = (p [ \Pi, r]_\mathsf{LY})_\mathrm{I} (\alpha, \beta)\\
        &= \big( p (\Pi \diamond r - r \diamond \Pi) \big)_\mathrm{I} (\alpha, \beta) \\
        &= \mathrm{pr}_\mathfrak{g} \big(  \Pi_\mathrm{I} (  ( r(\alpha), 0),  (0, \beta)) + \Pi_\mathrm{I} (   (0, \alpha), (r (\beta), 0)  )   \big) - r \circ \mathrm{pr}_\mathfrak{h} \big(  \Pi_\mathrm{I} ((0, \alpha), (0, \beta) )  \big)\\
        &= - \psi_\beta r (\alpha) + \psi_\alpha r(\beta) - r ([\alpha, \beta]')
    \end{align*}
    and similarly,
    \begin{align*}
        &(l_1 (r))_\mathsf{II} (\alpha, \beta, \gamma)
        = (p [ \Pi, r]_\mathsf{LY})_\mathrm{II} (\alpha, \beta, \gamma)\\
        &= \mathrm{pr}_\mathfrak{g} \big(  \Pi_\mathrm{II} (  (r(\alpha), 0), (0, \beta), (0, \gamma))  + \Pi_\mathrm{II} ((0, \alpha), (r(\beta), 0), (0, \gamma)) + \Pi_\mathrm{II} ( (0, \alpha), (0, \beta), (r (\gamma), 0))   \big) \\
        & \qquad \qquad \qquad - r \circ \mathrm{pr}_\mathfrak{h} \big(  \Pi_\mathrm{II} ( (0, \alpha), (0 , \beta), (0, \gamma))  \big) \\
        &=  \nu (\beta, \gamma) r (\alpha) - \nu (\alpha, \gamma) r(\beta) + D_{\psi, \nu} (\alpha, \beta) r (\gamma) - r (\llbracket \alpha, \beta, \gamma \rrbracket').
    \end{align*}
    Also, we have
    \begin{align*}
        &(l_2 (r, r))_\mathrm{I} (\alpha, \beta) = ( p [ [\Pi, r]_\mathsf{LY}, r ]_\mathsf{LY})_\mathrm{I} (\alpha, \beta)\\
        &= \mathrm{pr}_\mathfrak{g} ( [\Pi, r]_\mathsf{LY} \diamond r - r \diamond [\Pi, r]_\mathsf{LY})_\mathrm{I} (\alpha, \beta) \\
        &= \mathrm{pr}_\mathfrak{g} ([\Pi, r]_\mathsf{LY})_\mathrm{I} ( (r(\alpha), 0) , (0, \beta)) + \mathrm{pr}_\mathfrak{g} ([\Pi, r]_\mathsf{LY})_\mathrm{I} ( (0, \alpha), (r (\beta), 0) ) - r \circ \mathrm{pr}_\mathfrak{h} \big(  ([\Pi, r]_\mathsf{LY})_\mathrm{I} ( (0, \alpha), (0, \beta)) \big) \\
        &= \mathrm{pr}_\mathfrak{g} \big(   \Pi_\mathrm{I} ((r(\alpha), 0), (r (\beta), 0) )  \big) - r \circ \mathrm{pr}_\mathfrak{h} \big(  \Pi_\mathrm{I} ( (r(\alpha), 0), (0, \beta))  \big) + \mathrm{pr}_\mathfrak{g} \big(   \Pi_\mathrm{I} ((r(\alpha), 0), (r (\beta), 0) )  \big) \\
        & \qquad - r \circ \mathrm{pr}_\mathfrak{h} \big(  \Pi_\mathrm{I} ((0, \alpha), (r (\beta), 0) )  \big) - r \circ \mathrm{pr}_\mathfrak{h} \big(  \Pi_\mathrm{I}  ( (r(\alpha), 0), (0, \beta)) + \Pi_\mathrm{I}   ((0, \alpha), (r (\beta), 0) )    \big) \\
        &= 2 \big(  [r(\alpha), r(\beta)] - r ( \rho_{r (\alpha)} \beta - \rho_{ r (\beta)} \alpha )  \big),
    \end{align*}
     \begin{align*}
       & (l_2 (r, r))_\mathsf{II} (\alpha, \beta, \gamma)  = ( p [ [\Pi, r]_\mathsf{LY}, r ]_\mathsf{LY})_\mathrm{II} (\alpha, \beta, \gamma)\\
       & =  \mathrm{pr}_\mathfrak{g} ( [\Pi, r]_\mathsf{LY} \diamond r - r \diamond [\Pi, r]_\mathsf{LY})_\mathrm{II} (\alpha, \beta, \gamma)\\
        &= \mathrm{pr}_\mathfrak{g} ([\Pi, r]_\mathsf{LY})_\mathrm{II} ( (r(\alpha), 0), (0, \beta), (0, \gamma) )  + \mathrm{pr}_\mathfrak{g} ([\Pi, r]_\mathsf{LY})_\mathrm{II} ( (0, \alpha), (r(\beta), 0), (0, \gamma)) \\
        & \qquad +   \mathrm{pr}_\mathfrak{g} ([\Pi, r]_\mathsf{LY})_\mathrm{II} ( (0, \alpha), (0, \beta), (r (\gamma), 0) ) - r \circ \mathrm{pr}_\mathfrak{h} \big(  ([\Pi, r]_\mathsf{LY})_\mathrm{II} ( (0, \alpha), (0, \beta), (0, \gamma)) \big) \\
        &= 0 \quad (\mathrm{by ~} (\ref{insert2}), (\ref{insert4}))
    \end{align*}
    and
     \begin{align*}
       & (l_3 (r, r, r))_\mathrm{I} (\alpha, \beta) = ( p [ [ [ \Pi, r]_\mathsf{LY}, r]_\mathsf{LY} , r]_\mathsf{LY})_\mathrm{I} (\alpha, \beta)\\
        &= \mathrm{pr}_\mathfrak{g} \big(   [[ \Pi, r]_\mathsf{LY}, r]_\mathsf{LY} \diamond r - r \diamond   [[ \Pi, r]_\mathsf{LY}, r]_\mathsf{LY}  \big)_\mathrm{I} (\alpha, \beta) \\
        &= \mathrm{pr}_\mathfrak{g} \big\{  ( [[ \Pi, r]_\mathsf{LY}, r]_\mathsf{LY})_\mathrm{I} ( (r(\alpha), 0), (0, \beta) ) +  ([[ \Pi, r]_\mathsf{LY}, r]_\mathsf{LY})_\mathrm{I} ((0, \alpha), (r(\beta), 0) )      \big\} \\
        & \qquad - r \circ \mathrm{pr}_\mathfrak{h} \big(   ( [[ \Pi, r]_\mathsf{LY}, r]_\mathsf{LY})_\mathrm{I} ((0, \alpha), (0, \beta) )   \big) \\
        &= 0 \quad (\mathrm{by ~} (\ref{insert1}), (\ref{insert3})),
    \end{align*}
     \begin{align*}
        &(l_3 (r,r, r))_\mathsf{II} (\alpha, \beta, \gamma)\\
        &= \mathrm{pr}_\mathfrak{g} \big(   [[ \Pi, r]_\mathsf{LY}, r]_\mathsf{LY} \diamond r - r \diamond   [[ \Pi, r]_\mathsf{LY}, r]_\mathsf{LY}  \big)_\mathrm{II} (\alpha, \beta, \gamma) \\
        &= \mathrm{pr}_\mathfrak{g} \big\{  \big( [[ \Pi, r]_\mathsf{LY}, r]_\mathsf{LY} \big)_\mathrm{II} (  (r (\alpha), 0), (0, \beta), (0, \gamma)) +   \big( [[ \Pi, r]_\mathsf{LY}, r]_\mathsf{LY} \big)_\mathrm{II} ((0, \alpha), (r (\beta), 0), (0, \gamma) )  \\
        & \quad +   \big( [[ \Pi, r]_\mathsf{LY}, r]_\mathsf{LY} \big)_\mathrm{II} ((0, \alpha), (0, \beta), (r (\gamma), 0) )  \big\} - r \circ \mathrm{pr}_\mathfrak{h} \big( [[ \Pi, r]_\mathsf{LY}, r]_\mathsf{LY} \big)_\mathrm{II} ((0, \alpha), (0, \beta), (0, \gamma )) \\
        &= 2 ~ \! \mathrm{pr}_\mathfrak{g} \big\{ \big( [\Pi, r]_\mathsf{LY} \big)_\mathrm{II} ( (r(\alpha), 0), (r (\beta), 0), (0, \gamma)) + \big( [\Pi, r]_\mathsf{LY} \big)_\mathrm{II} ( (r(\alpha), 0), (0, \beta), (r( \gamma), 0) )   \\
        & \quad + \big( [\Pi, r]_\mathsf{LY} \big)_\mathrm{II} ( ( 0, \alpha), (0, \beta), (r( \gamma), 0)) \big\} - 2 ~ \! r \circ \mathrm{pr}_\mathfrak{h} \big\{ \big( [\Pi, r]_\mathsf{LY} \big)_\mathrm{II} ( (r(\alpha), 0), (0, \beta), (0, \gamma))  \\
       & \quad +  \big( [\Pi, r]_\mathsf{LY} \big)_\mathrm{II} ( ( 0, \alpha), (r(\beta), 0), (0, \gamma)) + \big( [\Pi, r]_\mathsf{LY} \big)_\mathrm{II} ( ( 0, \alpha), ( 0, \beta), (r (\gamma), 0) ) \big\} \\
        &= 6 ~ \! \big\{  \llbracket r(\alpha), r(\beta), r (\gamma) \rrbracket - r \big( D_{\rho, \mu} (r (\alpha), r(\beta)) \gamma + \mu (r (\beta), r(\gamma)) \alpha - \mu (r(\alpha), r (\gamma)) \beta \big)  \big\}.
    \end{align*}
    As a result, we get that
    \begin{align*}
        \big(  l_1 (r) + \frac{1}{2 !} l_2 (r, r) + \frac{1}{3!} l_3 (r,r,r)  \big)_\mathrm{I} (\alpha, \beta) = 0 \text{ if and only if } r \text{ satisfies the identity } (\ref{defor-1}),\\
        \big(  l_1 (r) + \frac{1}{2 !} l_2 (r, r) + \frac{1}{3!} l_3 (r,r,r)  \big)_\mathrm{II} (\alpha, \beta, \gamma) = 0 \text{ if and only if } r \text{ satisfies the identity } (\ref{defor-2}).
       \end{align*}
   Hence $r$ is a Maurer-Cartan element (i.e., $ l_1 (r) + \frac{1}{2 !} l_2 (r, r) + \frac{1}{3!} l_3 (r,r,r) = 0$) if and only if $r$ is a deformation map.
\end{proof}

The above theorem can be regarded as the Maurer-Cartan characterization of deformation maps in a given matched pair of Lie-Yamaguti algebras. As a consequence, we obtain the following result.

\begin{thm}\label{thm-governing}
    Let $(\mathfrak{g}, \mathfrak{h}, (\rho, \mu), (\psi, \nu))$ be a matched pair of Lie-Yamaguti algebras and $r : \mathfrak{h} \rightarrow \mathfrak{g}$ be a deformation map on it. Then $(\mathfrak{a} = C^{\bullet +1} (\mathfrak{h}, \mathfrak{g}), \{ l_k^r \}_{k=1}^\infty)$ is a new $L_\infty$-algebra, where
    \begin{align*}
        l_1^r (P) =~& l_1 (P) + l_2 (r, P) + \frac{1}{2} l_3 (r,r,P), \\
        l_2^r (P, Q) =~& l_2 (P, Q ) + l_3 (r, P, Q), \\
        l_3^r (P, Q, R) =~& l_3 (P, Q, R)\\
        \text{ and } l_k^r =~& 0, \text{ for } k \geq 4.
    \end{align*}
    Moreover, for any other linear map $r' : \mathfrak{h} \rightarrow \mathfrak{g}$, the sum $r+ r '$ is still a deformation map if and only if $r'$ is a Maurer-Cartan element of this new $L_\infty$-algebra $(\mathfrak{a} = C^{\bullet +1} (\mathfrak{h}, \mathfrak{g}), \{ l_k^r \}_{k=1}^\infty)$.
\end{thm}

\begin{proof}
    It follows from Theorem \ref{thm-mc-char} that $r \in \mathfrak{a}_0$ is a Maurer-Cartan element of the $L_\infty$-algebra $(\mathfrak{a}, \{ l_k \}_{k=1}^\infty)$. Hence, by Getzler \cite{getzler}, one can construct a new $L_\infty$-algebra structure on the same underlying graded vector space $\mathfrak{a}$ twisted by the Maurer-Cartan element $r$. The structure operations are precisely the ones given in the statement.

    For the second part, we observe that
    \begin{align}
        &l_1 (r+ r') + \frac{1}{2!} l_2 (r+r', r+r') + \frac{1}{3!} l_3(r+r' , r+r', r+r') \nonumber \\
        &= \underbrace{ l_1 (r) + \frac{1}{2!} l_2 (r,r) + \frac{1}{3!} l_3 (r,r,r)}_{= 0} + \big\{  l_1 (r') + l_2 (r,r') + \frac{1}{2} l_3 (r,r,r') \nonumber \\
        & \qquad \qquad \qquad  + \frac{1}{2} l_2 (r',r') + \frac{1}{2} l_3 (r,r',r') + \frac{1}{3!} l_3 (r',r',r') \big\} \nonumber \\
        &= l_1^r (r') + \frac{1}{2} l_2^r (r', r') + \frac{1}{3!} l_3^r (r',r',r'). \label{mc-deform}
    \end{align}
    This shows that $r+r'$ is a deformation map if and only if the expression given in (\ref{mc-deform}) vanishes, i.e., $r'$ is a Maurer-Cartan element of $(\mathfrak{a}, \{ l_k^r \}_{k=1}^\infty)$.
\end{proof}

The above result shows that the $L_\infty$-algebra $(\mathfrak{a} = C^{\bullet +1} (\mathfrak{h}, \mathfrak{g}), \{ l_k^r \}_{k=1}^\infty)$ governs the linear deformations of the operator $r$. For this reason, this $L_\infty$-algebra is said to be the {\em governing algebra} of the deformation map $r$. On the other hand, since $(\mathfrak{a} = C^{\bullet +1} (\mathfrak{h}, \mathfrak{g}), \{ l_k^r \}_{k=1}^\infty)$ is an $L_\infty$-algebra, it follows that the degree $1$ map $l_1^r : \mathfrak{a} \rightarrow \mathfrak{a}$ satisfies $(l_1^r)^2 = 0$. Hence, one may associate a cochain complex
\begin{align}\label{last-complex}
    C^1 (\mathfrak{h}, \mathfrak{g}) \xrightarrow{l_1^r} C^2 (\mathfrak{h}, \mathfrak{g}) \xrightarrow{l_1^r}  \cdots \cdots \xrightarrow{l_1^r} C^n (\mathfrak{h}, \mathfrak{g}) \xrightarrow{l_1^r} C^{n+1} (\mathfrak{h}, \mathfrak{g}) \xrightarrow{l_1^r} \cdots.
\end{align}
Hence, for any $n \geq 2$, the $n$-th cohomology group $H^n_r (\mathfrak{h}, \mathfrak{g})$ of the deformation map $r$ (defined in Subsection \ref{subsection-51}) is isomorphic to the $n$-th cohomology group of the cochain complex given in (\ref{last-complex}).


\medskip

  \noindent {\bf Acknowledgements.}  The author would like to thank the Department of Mathematics, IIT Kharagpur, for providing a beautiful academic atmosphere where the research has been done.
    
\medskip

\noindent {\bf Conflict of interest statement.} There is no conflict of interest.

\medskip

\noindent {\bf Data Availability Statement.} Data sharing does not apply to this article as no new data were created or analyzed in this study.


\begin{thebibliography}{BFGM03}

\bibitem{agore-ass} A. L. Agore, Classifying complements for associative algebras, {\em Linear Algebra Appl.} 446 (2014), 345-355.

\bibitem{agore-lie} A. L. Agore and G. Militaru, Classifying complements for Hopf algebras and Lie algebras, {\em J. Algebra} 391 (2013), 193-208.

\bibitem{agore-leibniz} A. L. Agore and G. Militaru, Unified products for Leibniz algebras. Applications, {\em Linear Algebra Appl.} 439 (2013), 2609-2633.

\bibitem{benito0} P. Benito, C. Draper and A. Elduque, Lie-Yamaguti algebras related to $\mathfrak{g}_2$, {\em J. Pure Appl. Algebra} 202 (2005), 22-54.

\bibitem{benito1} P. Benito, A. Elduque and F. Mart\'{i}n-Herce, Irreducible Lie-Yamaguti algebras, {\em J. Pure Appl. Algebra} 213 (2009), 795-808.

\bibitem{das-ms} A. Das, Associative-Yamaguti algebras, arXiv:2509.03648 [math.RA].

\bibitem{getzler} E. Getzler, Lie theory of nilpotent $L_\infty$-algebras, {\em Ann. Math. (2)} 170 (2009), 271-301.

\bibitem{goswami} S. Goswami, S. K. Mishra and G. Mukherjee, Automorphisms of extensions of Lie-Yamaguti algebras and inducibility problem, {\em J. Algebra} 641 (2024), 268-306.

\bibitem{jacobson} N. Jacobson, Lie and Jordan triple systems, {\em Amer. J. Math.} 71 (1949), 149-170.

\bibitem{jiang-sheng-tang} J. Jiang, Y. Sheng and R. Tang, Deformation maps of quasi-twilled Lie algebras, arXiv:2405.02532 [math-ph].

\bibitem{kikkawa} M. Kikkawa, Remarks on solvability of Lie triple algebras, {\em Mem. Fac. Sci. Shimane Univ.} 13
(1979), 17–22.

\bibitem{kinyon} M. K. Kinyon and A. Weinstein, Leibniz algebras, Courant algebroids, and multiplications on reductive homogeneous spaces, {\em Amer. J. Matth.} 123 (2001), 525-550.

\bibitem{lister} W. G. Lister, A structure theory of Lie triple systems, {\em Trans. Amer. Math. Soc.} 72, no. 2 (1952), 217-242.

\bibitem{majid} S. Majid, Physics for algebraists: non-commutative and non-cocommutative Hopf algebras by a bicrossproduct construction, {\em J. Algebra} 130 (1990), 17-64.

\bibitem{meyberg} K. Meyberg, Lectures on algebras and triple systems, Lecture notes, Charlottesville VA, University of Virginia (1972).

\bibitem{mondal-saha} B. Mondal and R. Saha, Deformation cohomology of morphisms of Lie-Yamaguti algebras, {\em J. Algebra} 658 (2024), 660-685.

\bibitem{nomizu} K. Nomizu, Invariant affine connections on homogeneous spaces, {\em Amer. J. Math.} 76 (1954), 33-65.

\bibitem{sagle} A. A. Sagle, On simple algebras obtained from homogeneous general Lie triple systems, {\em Pacific J. Math.} 15 (1965), 1397–1400.

\bibitem{sagle2} A. A. Sagle, A note on simple anti-commutative algebras obtained from reductive homogeneous spaces, {\em Nagoya Math. J.} 31 (1968), 105-124.

\bibitem{sheng-zhao} Y. Sheng and J. Zhao, Relative Rota-Baxter operators and symplectic structures on Lie-Yamaguti algebras, {\em Comm. Algebra} 50 (2022), 4056-4073.

\bibitem{smirnov} O. Smirnov, Imbedding of Lie triple systems into Lie algebras, {\em J. Algebra} 341 (2011), 1-12.

\bibitem{sun-chen} Q. Sun and S. Chen, Representations and cohomologies of differential Lie-Yamaguti algebras with any weights, {\em J. Lie Theory} 33 (2023), No. 2, 641-662.

\bibitem{takahashi} N. Takahashi, Representations of Lie-Yamaguti algebras with semisimple enveloping Lie algebras, {\em J. Algebra} 664 (2025), 452-483.

\bibitem{voro} Th. Voronov, Higher derived brackets and homotopy algebras, {\em J. Pure Appl. Algebra} 202 (2005), 133-153.

\bibitem{yamaguti0} K. Yamaguti, On the Lie triple system and its generalization, {\em  J. Sci. Hiroshima Univ. Ser. A} 21, no. 3 (1958), 155-160.

\bibitem{yamaguti} K. Yamaguti, On cohomology groups of general Lie triple systems, {\em Kumamoto J. Sci. Ser. A} 8, no. 4 (1969), 135-146.

\bibitem{zhang-li} T. Zhang and J. Li, Deformations and extensions of Lie-Yamaguti algebras, {\em Linear Multilin. Algebra} 63, no. 11 (2015), 2212-2231.

\bibitem{zhao-qiao} J. Zhao and Y. Qiao, Maurer-Cartan characterization, $L_\infty$-algebras, and cohomology of relative Rota-Baxter operators on Lie-Yamaguti algebras, arXiv:2310.05360 [math.RA].

\bibitem{zhao-qiao2} J. Zhao and Y. Qiao, The classical Lie-Yamaguti Yang-Baxter equation and Lie-Yamaguti bialgebras, {\em Sci. Sin Math.} 53 (2023), 1303-1324.

\bibitem{zhao-xu-qiao} J. Zhao, S. Xu and Y. Qiao, Post Lie-Yamaguti algebras, relative Rota-Baxter operators of nonzero weights, and their deformations, arXiv:2304.06324 [math.RA].
\end{thebibliography}
\end{document}